\documentclass[a4paper, 12pt, one column, aas_macros]{amsart}
\usepackage[english]{babel}
\usepackage[english]{babel}
\usepackage{graphicx}
\usepackage[latin1]{inputenc}
\usepackage{amsmath,amssymb,hyperref,srcltx}
\usepackage{amsfonts}%
\usepackage[dvips]{geometry}
\usepackage[normalem]{ulem}

\newcommand{\diff}[2]{\mbox{{\rm Diff}{${\,}_{#1}({\mathbb C}^{#2},0)$}}}
\newcommand{\fdiff}[2]{\mbox{$\widehat{\rm Diff}{{\,}_{#1}({\mathbb C}^{#2},0)}$}}

\newcommand{\px}{\frac{\partial}{\partial x}} 
\newcommand{\pz}{\frac{\partial}{\partial z}} 

\newtheorem{pro}{Proposition}

\newtheorem{cor}{Corollary}
\newtheorem{rem}{Remark}
\usepackage{xwatermark}
\usepackage{xcolor}

\newtheorem{lemma}{Lemma}[section]
\newtheorem{propo}{Proposition}

\newtheorem*{theorem*}{Theorem}
\newtheorem{theorem}{Theorem}
\newtheorem*{main theorem*}{Fixed Point Curve Theorem}
\newtheorem*{P-D*}{Poincar\'{e}-Dulac normal form}
\newtheorem*{hsmtb*}{Holomorphic Stable Manifold Theorem for Diffeomorphisms}
\newtheorem*{sierpinski*}{Sierpi\'nski Theorem}
\newtheorem*{hsmtv*}{Holomorphic Stable Manifold Theorem for Vector Fields}
\newtheorem{example}{Example}

\theoremstyle{remark}
\newtheorem{defi}{Definition}

\newcommand{\C}{\mathbb{C}}

\newcommand{\dis}{\displaystyle}

\newcommand{\R}{\mathbb{R}}
\newcommand{\Z}{\mathbb{Z}}
\newcommand{\N}{\mathbb{N}}
\newcommand{\Q}{\mathbb{Q}}
\newcommand{\Sing}{\mathrm{Sing}}

\newcommand{\vf}[2]{ {\mathfrak X}{\,}_{#1}({\mathbb C}^{#2},0)}
\newcommand{\fvf}[2]{ {\hat{\mathfrak X}}{\,}_{#1}({\mathbb C}^{#2},0)}

\newcommand{\mfx}{{\mathfrak X}}
\newcommand{\re}{\mbox{ Re}}
\newcommand{\im}{\mbox{ Im }}
\newcommand{\difn}{\mbox{Diff}(\C^n,0)}

\newcommand{\ordem}{\mathrm{ord}}

\newcommand{\spec}{\mathrm{Spec}}
\newcommand{\id}{\mathrm{id}}
\title[A Fixed point curve 
theorem for finite orbits diffeomorphisms]{A Fixed point curve theorem for finite orbits local diffeomorphisms}
\author{Lucivanio Lisboa}
\email{lucivaniolisboa@gmail.com}

\author{Javier Rib\'{o}n}
\address{Instituto de Matem\'{a}tica e Estat\'\i stica \\
Universidade Federal Fluminense\\
Campus do Gragoat\'a\\
Rua Marcos Valdemar de Freitas Reis s/n, 24210\,-\,201 Niter\'{o}i, Rio de Janeiro - Brasil }

\email{jribon@id.uff.br}
 
\thanks{MSC-class. Primary:Primary: 32H50, 37C25; Secondary: 37F75, 34M25} 
\begin{document}

\begin{abstract}

We study local biholomorphisms with finite orbits in some neighborhood of the origin since they are intimately related to holomorphic foliations with closed leaves. We describe the structure of the set of periodic points in dimension 2. As a consequence we show that given a local biholomorphism $F$, in dimension 2 with finite orbits, there exists an analytic curve passing through the origin and contained in the fixed point set of some non-trivial iterate of $F.$ As an application we obtain that at least one eigenvalue of the linear part of $F$ at the origin is a root of unity. Moreover, we show that such a result is sharp by exhibiting examples of local biholomorphisms, with finite orbits, such that exactly one of the eigenvalues   is a root of unity. These examples are subtle since we show they can not be embedded in one parameter groups.
\end{abstract}
\maketitle
\tableofcontents	

\section{Introduction}\label{sec:introduction}
Let $F: U \to V$ be a biholomorphism where $U$ and $V$ are open sets of ${\mathbb C}^{n}$ 
that 
contain the origin $0$ and $F(0) = 0$. Fixed $p\in U,$ we define $F^0(p)=p$ and, if $F^0(p),...,\ F^{j-1}(p)\in U$ for $ j>0,$ we define $F^{j}(p)=F(F^{j-1}(p)).$ Given $A\subset U\cap V$ and $p\in A,$ we define the \emph{positive orbit of $p$ by $F$ in $A$} as
\[O_{F, A}^{+}(p) = \{ F^{j}(p);\ F^{k}(p)\in A,\ 0\leq k\leq j\}, \]
the \emph{negative orbit of $p$ by $F$ in $A$} as
\[  O_{F, A}^{-}(p) =  O_{F^{-1}, A}^{+}(p)\]
and the \emph{orbit of $p$ by $F$ in $A$} as 
\[ O_{F, A}(p)= O_{F, A}^{+}(p) \cup  O_{F, A}^{-}(p),\]
where $F^{-1}$ denote the inverse of $F.$ We say that $F$ has the \textit{finite orbits property} in $A$ if $O_{F,A}(p)$ is a finite set for all $p\in A.$ In this case, 
$F$ has the finite orbits property in $B$ for any subset $B$ of $A$
since $O_{F,B}(p)\subset O_{F,A}(p)$  for all $p \in B$.
As a consequence, the finite orbits property can be defined
for $F \in \difn$ where $\difn$ is the group of germs of biholomorphism fixing the 
origin  $0\in \C^n$. 

\begin{defi}\label{defi.gn.fin.orb.}
Let $F\in \difn$ be a local biholomorphism. We say that $F$ has the \textit{finite orbits property} 
(and then we write $F\in\diff{<\infty}{n}$) if there exists a representative $F:U\to V$ and a neighborhood $ A\subset U\cap V$ of $0$ 
such that $F$ has the finite orbits property in $A.$
\end{defi}

Finite orbits local biholomorphisms appear in the study of foliations with closed leaves.
In \cite{mattei-moussu1980} Mattei and Moussu proved one of the most important theorems in the theory of holomorphic foliations, namely the topological characterization of the existence of a non-constant holomorphic first integral for germs of codimension 1 holomorphic foliations.
More precisely, they show that a singular holomorphic foliation $\mathcal F$ on $({\mathbb C}^{n},0)$ of codimension 1 has a first integral if and only if  the leaves of $\mathcal F$ are closed subsets of the complement of
the singular set and only finitely many of them accumulate at 0. 
A fundamental ingredient of the proof is that, in dimension 1, the finite orbits property is equivalent to periodicity. More precisely, they show that a biholomorphism $F\in \diff{}{}$ has the finite orbits property if and only if $F$ is a finite order element of the group $\diff{}{}.$ 
For dimension $n\geq 2,$ the equivalence does not hold. For example, the local biholomorphism $F(x,y) = (x, y+ x^{2})$ has the finite orbits property but is non-periodic.  

It is possible to recover  the equivalence periodicity $\leftrightarrow$ finite orbits by replacing the finite order property with stronger conditions and so to obtain, in dimension greater than $1,$ analogues of the topological criterium of Mattei--Moussu (\cite{rebelo-reis2015a}, \cite{camara-scardua2009}, and \cite{camara-scardua2017}). 
In spite of this, the following elementary problems, related to the finite orbits property, 
were open until now:

\begin{enumerate}
	\item Relation of finite orbits with the jacobian matrix $D_0F$;
	 \item description of the set of periodic points of $F \in\diff{<\infty}{n}$.
\end{enumerate}
We answer these questions in dimension two and provide partial answers for higher dimension. 
A natural question (that was open until now) is whether the finite orbits property for 
$F \in \difn$  implies the analogue for the linear part $D_0 F$ of $F$ at the origin. 
The next result provides the first counterexamples. 
\begin{theorem} \label{thm_irrat_bih_with_fin_orb}
	Suppose that $\lambda\in \C$ satisfies the Cremer condition and $n \geq 1$.
	Then there exists  a global biholomorphism $F\in\mathrm{Diff}(\C^{n+1})$ such that 
	$\spec(D_0 F)=\{\lambda,1\}$, the algebraic multiplicity of the eigenvalue $1$ 
	of $D_0 F$ is equal to $1$ and $F$
	has the finite orbits property in every set of the form ${\mathbb C^{n}} \times U,$ where
	$U\subset \C$ is a bounded open set. 
\end{theorem}
 Let us stress that  $F$ is a counterexample because, 
in the linear case, the finite orbits property occurs if and only if the spectrum of $D_0 F$ consists of 
roots of unity (cf. Proposition  \ref{proposition.car.fin.orb.lin.case}). 
Until now it was  known  that 
finite orbits local biholomorphisms $F$ satisfy that 
 the eigenvalues of $D_0 F$ have modulus $1$ by the Stable Manifold Theorem (cf. Corollary \ref{cor_fin.orb_imp_mod1}).

The next result gives an indication of why the examples of $F\in\diff{<\infty}{n}$ 
such that $\mathrm{spec} (D_0 F)$ is not
contained in the group of roots of unity were missing in the literature: 
there are no ``continuous" examples, 
i.e. where $F$ belongs to a one-parameter group. Let $\mfx(\C^n,0)$ denote the 
Lie algebra of 
singular local holomorphic vector fields at the origin.

\begin{theorem} 
\label{thm_flow_fin_orb}
	Let $X\in\mfx(\C^n,0)$ and let $F$ be the time 1 map of $X.$ Suppose that $F$ has the finite orbits property. Then it satisfies
	\[\spec (D_0F)\subset e^{2\pi i \Q}.\]
\end{theorem}

The existence of the examples  provided by Theorem \ref{thm_irrat_bih_with_fin_orb} 
has  consequences in the problem of geometrical realization of formal invariant curves. Indeed, it was proved in \cite{lopes-raissy-ribon-sanz2019} that when the multiplier of the restriction $F_{|\Gamma}$ of $F\in \diff{}{2}$ 
to a formal invariant curve $\Gamma$  is not an element of $e^{2 \pi i({\mathbb R} \setminus {\mathbb Q})}$ and $F$ is non-periodic then either the curve $\Gamma$ is convergent or there are invariant analytic sets asymptotic to $\Gamma$ 
 and consisting of stable orbits, i.e. orbits of points $p$ such that 
 $\lim_{k \to \infty} F^{k} (p) =0$. 
Our examples show that in general there is no  systematic approach to the geometrical realization of $\Gamma$ as a stable set if the multiplier of $F_{|\Gamma}$ is irrationally neutral, i.e. if it belongs to $e^{2 \pi i ({\mathbb R} \setminus {\mathbb Q})}$. This completes somehow the realization program in \cite{lopes-raissy-ribon-sanz2019} and \cite{lopez2020}.
More precisely, the examples provided by Theorem \ref{thm_irrat_bih_with_fin_orb} for dimension $2$
have a formal curve $\Gamma$ invariant by $F,$ such that the multiplier $\lambda$ of $F|_\Gamma$ 
belongs to $e^{2 \pi i ({\mathbb R} \setminus {\mathbb Q})},$ but $F$  
has no stable sets, since it has finite orbits.

 Theorem \ref{thm_irrat_bih_with_fin_orb} suggests that the finite orbits property is related
 to small divisors.  Note that the multiplier $\lambda$  in Theorem  
\ref{thm_irrat_bih_with_fin_orb} is very well approached by roots of unity since it is a Cremer number. 
Such a circumstance is not accidental; indeed we show, by applying a Theorem of P\"{o}schel  \cite{poschel1986}, 
that there is no $F\in \diff{<\infty}{2}$ such that $\spec (D_0F)$ contains a Bruno number  (Proposition \ref{pro:pos}).
In particular, we show that if the multiplier of $F_{|\Gamma}$ is a Bruno number, 
for $F\in \diff{}{2}$  and a formal invariant curve 
$\Gamma$ of $F$, then $F$ does not satisfy the finite orbits property (Corollary \ref{cor:pos}).

 The natural follow up question to Theorem  \ref{thm_irrat_bih_with_fin_orb}
is to understand how far is $D_0 F$ of satisfying the finite orbits property if 
$F \in \diff{<\infty}{n}$. In particular, are there $F \in  \diff{<\infty}{n}$ such that 
$\spec (D_0F) \cap  e^{2\pi i \Q} = \emptyset$?
The answer is  negative for dimension 2. 
\begin{theorem}\label{thm_at_least_one_root}
	Let $F\in \diff{<\infty}{2}$. Then at least one eigenvalue of $D_0F$ is a root of unity
	(and all of them  belong to the unit circle).
\end{theorem} 
As a consequence, the examples of Theorem \ref{thm_irrat_bih_with_fin_orb} have the 
minimal number of roots of unity eigenvalues and hence 
Theorem \ref{thm_at_least_one_root} is sharp. 
 Moreover, we classify the finite orbits local biholomorphisms $F\in \diff{}{2}$ 
such that $\spec (D_0F)$ contains a non-root of unity eigenvalue: essentially they are the examples
provided by Theorem  \ref{thm_irrat_bih_with_fin_orb}  (Proposition \ref{pro:pos}).
Theorem \ref{thm_at_least_one_root} is a consequence of 
 the Fixed Point Curve Theorem that we discuss 
next. A classical result about vector fields in dimension $n=2$ is the Camacho-Sad theorem \cite{camacho-sad1982}, which states that every vector field $X\in \mfx(\C^2,0)$ admits a germ of invariant curve at the origin.
Existence of invariant objects for local biholomorphisms tangent to the identity $F\in \diff{1}{2}$ is also well known (see \cite{abate2001}, \cite{lopes-sanz2015}, \cite{brochero-cano-lopez2008}....). 
In \cite{abate2001} Abate generalizes to $\C^2$ the classical Leau-Fatou flower theorem proving that if $F\in \diff{1}{2}$ has an isolated fixed point at $0$ then $F$ has at least one parabolic curve, that is, a $F$-invariant holomorphic curve, with the origin in their boundary, and 
whose orbits tend to $0$;
in particular, it does not have the finite orbits property. 
Thus, if $F\in \diff{1}{2}\cap\diff{<\infty}{2}$ then $F$ has  a non isolated fixed point at the origin.
Later on, L\'{o}pez-Hernanz and Sanz \cite{lopes-sanz2015} showed that if $F$ 
has a formal invariant curve $\Gamma$ that is not contained in the fixed point set of $F$ then
$F$ or $F^{-1}$ has a parabolic curve asymptotic to $\Gamma$. In particular, if 
$F\in \diff{1}{2}\cap\diff{<\infty}{2}$ then every formal invariant curve is a fixed point curve.

In contrast to the previous approach,
the second author showed that existence of germs of analytic invariant curves does not hold for biholomorphisms $F \in \diff{}{2}$ in general \cite{ribon2005}. Moreover, the counterexamples can be chosen to be tangent to the identity or formally linearizable. 

In this work we prove the following theorem. It is a version of  the Camacho-Sad 
Theorem for local biholomorphisms that satisfy the finite orbits property.
\begin{main theorem*}
	Let $F\in \diff{}{2}$ be a diffeomorphism with the finite orbits property. Then, there exists $m\in \N$ such that $F^m$ has a germ $\Gamma$ of complex analytic curve consisting of fixed points. 
\end{main theorem*}

We can apply The Fixed Point Curve Theorem to obtain a generalization of a Rebelo--Reis Theorem in the context of cyclic subgroups of $\diff{}{2}.$ More precisely, a consequence of Theorem $A$ in \cite{rebelo-reis2015a} is that if, for all $m\in \N,$ every point $p \in \mathrm{Fix}(F^{m})$ in a neighborhood of $0$ satisfies that either $p$ is an isolated fixed point of $F^{m}$ or the germ of $F^{m}$ at $p$ is equal to the identity map, then $F$ is periodic. 
We provide a stronger version of this result in dimension 2, namely it suffices to 
check the condition at the origin. 
Thus, we obtain a  negative criterium for the finite orbits property. 
\begin{cor}\label{cor_orbfin_imp_alg.crit}
	Let $F \in \mathrm{Diff} ({\mathbb C}^{2},0)$ such that $0$ is an isolated fixed point of $F^{m}$ for every $m \in {\mathbb N}$. Then $F$ does not satisfy the finite orbits property.
\end{cor}
Our  approach to show the Fixed Point Curve Theorem relies on
describing the connected components of the 
set of periodic points of  $F\in\diff{<\infty}{2}$.
\begin{theorem}
\label{thm:structure}
Let $F\in\diff{<\infty}{2}$. Let $B$ be an open or closed ball centered at the origin 
such that  $F$ and $F^{-1}$
are defined in a neighborhood $U$ of $\overline{B}$ and $F$ has finite orbits in $U$. 
Consider the sets 
\[ \mathrm{Per}_{k} (F) = \{ p \in B; \ 
p, \ F(p), \hdots, F^{k}(p) \in B \ \mathrm{and} \ F^{k}(p)=p \} \]
 for $k \in {\mathbb N}$. Let $C$ be a connected component of 
$\mathrm{Per} (F):=\cup_{k=1}^{\infty}  \mathrm{Per}_{k} (F)$. 
Then, $C$ is semianalytic and there exists $m=m(C)$ such that $C$ is a connected component of  
the semianalytic set $\mathrm{Per}_{m} (F)$. Moreover,  if $B$ is an open ball then
$C$ is complex analytic in $B$ and the irreducible components of $C$ have positive dimension.
\end{theorem}

 Suppose that $B$ is a closed ball since it is simpler to work in compact sets. 
Let ${\mathcal B}$ be the set of points $p \in B$ such that the map $q \mapsto \sharp O_{F,B} (q)$ is an unbounded function in every neighborhood of $p$. 
Such a set is the analogue for diffeomorphisms of the  so called {\it bad set} associated to smooth
foliations by compact leaves of compact manifolds; it consists of the leaves  where 
the volume function (defined in the space of leaves) is not locally bounded. 
The properties of the bad set are one of the ingredients 
used by Edwards, Millet and Sullivan to show that, under a suitable homological condition, 
the volume function associated to a smooth foliation by compact leaves is uniformly bounded 
\cite{edwards-millet-sullivan1975}. In the finite orbits case for $n=2$, the bad set ${\mathcal B}$ 
is contained in   $\mathrm{Per} (F)$ and moreover, the connected
components of ${\mathcal B}$ are also connected components of   $\mathrm{Per} (F)$.  
In general, the structure of the bad set can be very complicated. 
However, the finite orbits property constrains the connected components of the bad set 
${\mathcal B}$ to be simple for $n=2$. Indeed they are 
semianalytic by Theorem \ref{thm:structure}.

Our results can be used to study holomorphic foliations of codimension $2$ defined in a neighborhood of 
a compact leaf and whose leaves are closed. Such a problem will be contemplated in future work.

Section \ref{sec.1} introduces the setting of the paper along with some elementary results. 
Theorem \ref{thm_flow_fin_orb} is proved in section \ref{sec:one_parameter}.
We show Theorems \ref{thm:structure}, \ref{thm_at_least_one_root} and the 
Fixed Point Curve Theorem in section  \ref{sec:fixed}.
Finally, we provide the examples in Theorem \ref{thm_irrat_bih_with_fin_orb}
in section \ref{sec:non-virtual}.

\section{Notations and first results}\label{Notations and first results}\label{sec.1} 

As above, we denote by $\diff{}{n}$ the group of germs of biholomorphisms fixing $0\in \C^n$ and by $\diff{<\infty}{n}$ the subset of $\diff{}{n}$ consisting of those that have the finite orbits property. In the remainder of this section we included some elementary results about the finite orbits property for the sake of completeness.

\begin{propo}\label{propo_equiv_of_fin.orb}\label{lemma.gn.inv.fin.orb.}\emph{Let $F\in \difn.$ 
		\begin{itemize}
			\item[(i)] (Invariance by analytic conjugation) If $F=HGH^{-1}$ for some $H\in \diff{}{n},$  then 
			$$F\in \diff{<\infty}{n} \Leftrightarrow G\in \diff{<\infty}{n}.$$
			\item[(ii)] (Invariance by iteration) The following affirmations are equivalent.
			\begin{enumerate}
				\item[(a)] $F\in \diff{<\infty}{n}.$ 
				\item[(b)] $F^m\in \diff{<\infty}{n},\ \forall m\in \N.$
				\item[(c)] $F^m \in\diff{<\infty}{n}$ for some   $m\in \N.$
			\end{enumerate}
	\end{itemize}}
\end{propo}
\begin{proof}
	(i) Assume $H\circ G=F\circ H$ and let $U$ be a connected neighborhood of $0$ in which all germs involved have injective representatives. There exists a neighborhood $0\in A\subset U$ such that $G(A),H(A),H(G(A))\subset U.$ By using $H\circ G=F\circ H$, we can show by induction that, if $x,G^{\pm}(x),...,G^{\pm k}(x)\in A,$ then 
	$F^l(H(x))=H(G^l(x))$ for $l=\pm 1,...,\pm k.$ Therefore,  
	\[H (O_{G,A}(x))= O_{F, H(A)}(H(x))\ \mbox{ for any }x\in A.\]
	In particular, $G$ has the finite orbits property in $A$ if and only if $F$ has the finite orbits property in $H(A).$ This shows (i). 
	
	(ii) (a)$\Rightarrow$ (b): Suppose $F\in \diff{<\infty}{n}$ and let $U$ be a neighborhood of $0$ in which $F$ e $F^{-1}$ are defined and have finite orbits. Given $m\in \N,$ there exists a neighborhood $V$ of $0$ such that $F^{\pm j}(p)\in U$ for all $ p\in V$ and $0\leq j\leq m.$ 
	As a consequence, we obtain 
	\[O_{F^m,V}(p)\subset  O_{F,U}(p),\ \forall p\in V,\]
	 and hence $F^m\in \diff{<\infty}{n}$ for any $m\in \N.$ It is obvious that (b) $\Rightarrow$(c). Let us show that (c)$\Rightarrow$(a). Let $m\in \N$ such that $F^m\in \diff{<\infty}{n}$. As before, let $V$ be a neighborhood in which $F^m$ has finite orbits. So, for each $p\in V,$ either the set $\{ j \in {\mathbb Z} : F^{mj}(p) \in V\}$ is finite or $F^{mk}(p)=p$ for some $k \in {\mathbb Z}^{*}.$ In any case, we see that  $O_{F,V}(p)$ is a finite set; that is, $F\in \diff{<\infty}{n}.$  This concludes the proof.
\end{proof}

We say that a biholomorphism $F\in \difn$ is periodic if it is a finite order element of the group $\mathrm{Diff}({\mathbb C}^{n},0),$ that is, if there exists $m\in \N$ such that $F^m=\id.$ In dimension $n=1,$ Mattei and Moussu proved in \cite{mattei-moussu1980} 
that the finite orbits property is equivalent to periodicity. For dimension $n>1,$ the equivalence is far from be true. For example, the biholomorphism  
$F(x,y)=(x,x+y)$ is non-periodic, but has the finite orbits property in each bounded neighborhood of $0\in \C^2:$ the line $\{x=0\}$ is the set of fixed points of $F$ and $F$ is a non-trivial translation on $\{x=c\}$ with $c \neq 0.$

\begin{rem}\emph{The subset $\diff{<\infty}{}$ of $ \diff{}{} $ is not a subgroup. For example, the biholomorphisms $F(x)=-x$ and $G(x)=-\frac{x}{1-x}$ are periodic of period 2 and, hence, 
belong to $\diff{<\infty}{}.$ However, the composition 
		\[H(x)=(F\circ G)(x)=\frac{x}{1-x}\]
		is not periodic, since
		$H^n(x)=\frac{x}{1-nx},\ n\in \N.$ 
		On the other hand, it is easy to check that the subset of $\diff{}{}$ formed by the linear isomorphisms with finite orbits is 
		a subgroup of $\diff{}{},$ isomorphic to the group of roots of unity. Nevertheless, this does not hold for dimension greater than $1$ (cf. Corollary \ref{cor.of.nao.subg})}\end{rem}

\begin{propo}[Linear case] \label{proposition.car.fin.orb.lin.case}\label{propolin_fin.orb_iff_rat}
	Let $F\in \diff{}{n}$ be analytically linearizable. 
	Then $F$ has the finite orbits property if and only if its eigenvalues are roots of unity. Furthermore, if $m$ is the least positive integer such that $F^m$ is unipotent then every periodic point of $F$ is a fixed point of $F^m$.
\end{propo}

\begin{proof}
	Applying Proposition \ref{propo_equiv_of_fin.orb} (i), we can assume $F(x)=Ax,$ where $A\in\mathrm{GL}(n,\C).$ Let $\lambda\in \spec(A)$ and let 
	$B$ be an arbitrary ball centered at $0.$ If $v\in B$ is a eigenvector of $F$ associated with $\lambda,$ then $F^m(v)=\lambda^mv$ for all $m\in \Z.$ In particular, $O_{F,B}(v)$ is finite if and only if $\lambda$ is a root of unity. Hence,   the finite 
	orbits property implies that  all eigenvalues of $F$ are roots of unity.
	 
	Reciprocally, suppose that the eigenvalues of $F$ are roots of unity. It suffices to show that $\C ^n$ admits a decomposition $\C^n=V_1\oplus\cdots \oplus V_r$ as $F$-invariant subspaces such that, for each $x_j \in V_j,$ the orbit ${\mathcal O}_{F,U}(x_j)$ is finite for any bounded set $U$ and any $x_j\in U\cap V_j.$ Thus,  
	 it suffices to consider the case where $A$ is a Jordan block
	\[	\label{bloco de jordan n-dim} A = \left[
	\begin{array}{ccccc}
	\lambda &  &  & &\\
	1 & \lambda &  & &\\
	& 1 & \ddots & &\\
	&  & \ddots & \lambda& \\
	&  &    &1& \lambda\\
	\end{array}
	\right], \]
	where $\lambda$ is a root of unity.   We can assume $n>1$ since the remaining case is 
	trivial. 
	Let $x=(x_1,...,x_n)\in\C^n.$ 
	By using induction on $m$ it is easy to see that
	\[F^m(x)=(\lambda^mx_1,m\lambda ^{m-1}x_1+\lambda^m x_2,\ 
	... ,\	\ P_{n,m}(x_1,...,x_{n-2})+m\lambda ^{m-1}x_{n-1}+\lambda^mx_n),\]
	with $P_{j,m}$ linear for all $j,m\in\N.$ 
	In other words, if $\pi_j$ is the $j$th projection on $\C^n,$ then
	\[\begin{array}{cclcc}
	\pi_1(F^m(x)) & = & \lambda^m x_1,  &\\
	\pi_2(F^m(x)) & = & m\lambda ^{m-1}x_1+\lambda^m x_2, & &\\
	\pi_3(F^m(x)) & = & P_{3,m}(x_1)+m\lambda^{m-1}x_{2}+\lambda ^mx_3 & \\
	& \vdots &  &  \\
	\pi_n(F^m(x)) & = & P_{n,m}(x_1,...,x_{n-2})+m\lambda ^{m-1}x_{n-1}+\lambda^mx_n.   &\\
	\end{array} \]
	Consider $x \neq 0.$ Let $j_0$ be
	 the first index such that $x_{j_0} \neq 0$. If $j_0 = n$ then $x$ is periodic;  so it has a finite orbit. Consider $j_0 < n.$ Then
	$$|\pi_{j_0+1}(F^m (x))|=|m\lambda^{m-1}x_{j_0}+\lambda ^mx_{j_0+1}|\to \infty\ \mbox{ when }\ m\to \infty.$$
	In any case, we see that $x$ has finite positive orbit in any bounded neighborhood of $0.$ Since the eigenvalues of $A^{-1}$ are the inverses of the eigenvalues of $A,$ we show that every negative orbit of $x$ is finite analogously. This shows that $F$ has finite orbits. Moreover, if $m$ is the least positive integer such that $F^m$ is unipotent, the discussion above shows that for a Jordan block the set of periodic points coincide with $\mathrm{Fix}(F^m).$ Therefore, $\mathrm{Fix}(F^m)$ is the set of periodic points of $F$.
\end{proof}

\begin{cor} \label{cor.of.nao.subg}
	Suppose $n \geq 2$. Then the subset of $\diff{}{n}$ consisting of all linear biholomorphisms with finite orbits is not a subgroup of  $\diff{}{n}.$
\end{cor}

\begin{proof}
	Let
	\[F(x_1,...,x_n)=(x_1,x_1+x_2,x_3,...,x_n)\ \mbox{ and }  \ G(x_1,...,x_n)=(x_1+x_2,x_2,x_3,....,x_n). \]
	Then $F$ and $G$ are linear and have the finite orbits property, since they are unipotent. However, $F\circ G=(x_1+x_2,x_1+2x_2,x_3,...,x_n)\notin \diff{<\infty}{n},$ since 
	 $\frac{3 + \sqrt{5}}{2}$ and $\frac{3 - \sqrt{5}}{2}$ are eigenvalues of $F \circ G$.
\end{proof}
\strut 
\subsection{Formal diffeomorphisms}
We denote by ${\mathcal O}_{n,0}$ (resp. $\hat{\mathcal O}_{n,0}$) the local ring of 
convergent (resp. formal) power series with complex coefficients in $n$ variables centered at the origin
of ${\mathbb C}^{n}$. Let ${\mathfrak m}$ be the maximal ideal of  $\hat{\mathcal O}_{n,0}$.
\begin{defi}
The group $\fdiff{}{n}$ of formal diffeomorphisms consists of the elements 
$F = (F_1, \hdots, F_n)$  of 
${\mathfrak m} \times \hdots \times {\mathfrak m}$ such that its first jet $D_0 F$ belongs to 
$\mathrm{GL}(n, {\mathbb C})$. 
\end{defi}
\begin{rem}
We can identify $\diff{}{n}$ with the subset of $\fdiff{}{n}$ of formal diffeomorphisms 
$F=(F_1, \hdots, F_n)$ such that $F_j \in {\mathfrak m} \cap {\mathcal O}_{n,0}$ for 
every $1 \leq j \leq n$ as a consequence of the inverse function theorem.
\end{rem}
\begin{defi}
 We denote by $\fdiff{u}{n}$ the set of {\it unipotent} formal biholomorphisms, that is, the elements $F\in \fdiff{}{n}$ such that $\spec(D_0F)=\{1\}$.  Its subset 
$\fdiff{1}{n} := \{ F \in \fdiff{}{n} : D_0 F = \mathrm{id} \}$ is called the group of 
{\it tangent to the identity} formal diffeomorphisms in $n$ variables.  
\end{defi}
\subsection{Vector fields and flows}\label{sec.campos.e.fluxos}
Let $\mfx(\C^n,0)$ denote the Lie algebra of singular
local holomorphic vector fields at $0\in \C^n.$  
Let $X\in \mfx(\C^n,0)$  
and denote by $\phi(z,t)$ its local flow, defined in a neighborhood of $\{ 0 \} \times {\mathbb C}^{n}$ 
in ${\mathbb C}^{n+1}$.  
Then, for each $t\in {\mathbb C}$, the map $z\mapsto \phi^t(z)$ 
defines a biholomorphism $\phi^t\in\difn,$ 
the so called   {\it time-$t$}  
map of $X,$ also denoted by $\exp(tX).$ 
We also use  $\exp(1X)=\exp(X)$ to denote the time-$1$ map of $X.$ It turns out that if $f\in{\mathcal O}_{n,0}$ then 
\[f\circ\exp(tX)(z)=f(z)+\sum_{j=1}^{\infty}\frac{t^j}{j!}X^j(f)\]
 by Taylor's formula,  
where $X$ is now understood as a derivation in the ring ${\mathcal O}_{n,0}$ and $X^j=X\circ X^{j-1}.$ By considering $f=z_j,\ j=1,...,n,$ we obtain
\[\exp(tX)(z_1,...,z_n)=\left(z_1+\sum_{j=1}^{\infty}\frac{t^j}{j!}X^j(z_1),\ ...\ ,z_n+\sum_{j=1}^{\infty}\frac{t^j}{j!}X^j(z_n)\right).\]
This last identity allows us to extend the definition of flow associated to a germ of holomorphic singular vector field to a formal singular vector field $X\in\widehat{\mfx}(\C^n,0),$ i.e. a derivation of the ring of formal power series that preserves the maximal ideal. 
Indeed $\exp (t X) \in \fdiff{}{n}$ for all $X \in \widehat{\mfx}(\C^n,0)$ and $t \in {\mathbb C}$.
We say two vector fields $X,Y\in\vf{}{n}$ are analytically equivalent if there exists  $H\in \difn$ such that $H_*X=Y,$ that is 
\[D_xH \cdot X(x)=Y(H (x))\]
 for any $x$ in a neighborhood of the origin. 
The map $H$ is called a   \textit{conjugacy} 
between $X$ and $Y.$ In this case,  one could show that $H$ is also a 
  \textit{conjugacy}   between their time-$t$ maps for any $t \in {\mathbb C}.$ 

\begin{defi}\emph{We say that a singular vector field $X\in \vf{}{n}$ has the finite orbits property (denoting by $X\in \vf{<\infty}{n}$) if  $\exp(X)\in \diff{<\infty}{n}$. }
\end{defi}

Now we derive some properties of finite orbits vector fields. 
\begin{cor}
	Let $X\in\vf{}{n}$ be a singular vector field.
	\begin{itemize}
		\item[(i)] If $Y\in\vf{}{n}$ is singular and analytically conjugated to $X$ then
		\[ X \in\vf{<\infty}{n} \Leftrightarrow Y\in \vf{<\infty}{n}.\]
		\item [(ii)] The following statements are equivalents.
			\begin{enumerate}
			\item[(a)] $X$ has the finite orbits property. 
			\item[(b)] $mX$ has the finite orbits property for all $ m\in \N.$
			\item[(c)] $mX$ has the finite orbits property for some $m\in \N.$
		\end{enumerate}
	\end{itemize}   \end{cor}
\begin{proof}
	(i): As $\exp(X)$ is conjugated to $\exp(Y)$, we can apply Proposition \ref{propo_equiv_of_fin.orb} (i). 
	
	(ii) Set $F=\exp(X)\in \difn.$ Then 
	\[F^m=\exp(mX)\ \mbox{ for all }\ m\in \Z.\] 
	Then we can apply Proposition \ref{lemma.gn.inv.fin.orb.} (ii).
	\end{proof}

\begin{rem}\emph{
	The last corollary implies that the finite orbits property 
	holds for  $\Z$-multiples of $X$.
	However, in general it is not satisfied for $\R$ or $\C$-multiples of $X$.
	In fact, the time-$1$ map of the one-dimensional field $X=\lambda x\px$ is $f=e^\lambda x,\ \lambda\in \C.$ Since in dimension one finite orbits is equivalent to finite order, it follows that \[X\in \vf{<\infty}{}\Leftrightarrow f \mbox{ has finite order }\Leftrightarrow \lambda\in2\pi i\Q.\]}
\end{rem}

\begin{cor}\label{cor.of.sse.r.campo}
	Let $X\in\mfx(\C^n,0)$ be  an analytically  linearizable vector field. Then
	\[X \in\vf{<\infty}{n}\Leftrightarrow \spec(D_0X)\subset 2\pi i \Q.\]
\end{cor}
\begin{proof}
	The time-1 map $F:=\exp(X)$ is linearizable and its eigenvalues have the form $e^\lambda$ with $\lambda\in \spec(D_0X).$ The proof now is a consequence of Proposition \ref{propolin_fin.orb_iff_rat}.
\end{proof}

\begin{example} \emph{Consider the vector field $\dis X=x(2\pi i+y)\px.$ Then 
we have  
		\[\exp(X)=(e^{2\pi i+y}x,y)=(e^yx,y). \] 
		Moreover, in each level $\{y=c\}$ with $c\notin 2\pi i \Q$ every point $x\neq 0$ has an 
		infinite orbit.  Since the linear part $X_0=2\pi i x\px$ has finite orbits, it follows that $X$ is not analytically linearizable.}
\end{example}
\strut

\subsection{Jordan-Chevalley decomposition of vector fields}
We say that a singular vector field $X\in\fvf{}{n}$ is semisimple if it is formally conjugated to a vector field of the form 
$\sum_{j=1}^{n}\lambda_jx_j\frac{\partial}{\partial x_j}.$ We say that
$X$ is nilpotent if the linear part of $X$ is nilpotent. Finally, we say that  
\[X=X_S+X_N\]
is the semisimple/nilpotent decomposition of $X$ if $X_S$ is semisimple, $X_N$ is nilpotent and $[X_N,X_S]=0.$ Every singular formal vector field $X$ admits a unique semisimple/nilpotent decomposition (cf. \cite{martinet1981}). We denote by  $\hat{\mfx}_N(\C^n,0)$ the subset of $\fvf{}{n}$ consisting of the formal nilpotent vector fields.
  \begin{propo}[cf. \cite{ecalle1975, martinet-ramis1983}]
	The image of $\hat{\mfx}_N(\C^n,0)$ by the exponential aplication is $\widehat{\mathrm{Diff}}_{u}(\C^n,0)$ and the map $\exp:\hat{\mfx}_N(\C^n,0)\to\widehat{\mathrm{Diff}}_{u}(\C^n,0)$ is a bijection.
\end{propo}
\strut

\subsection{Poincar\'{e}-Dulac normal form}

First, let us
recall that a point $\lambda=(\lambda_1,...,\lambda_n)$ of $\C^n$ 
is said  \textit{resonant} if there exists $m\in \N_0^n$ with $|m|\geq 2$ and some  $1\leq k\leq n$ such that  
\[\langle\lambda,m\rangle=\lambda_1m_1+\cdots +\lambda_nm_n=\lambda_k.\] 
This equation is called a resonance and its 
\textit{resonant monomial} is the vector field
\[F_{k,m}~=x^m e_k=(0,...,x_1^{m_1}\cdots x_n^{m_n},...0)\]
Note that $\lambda=(\lambda_1,...,\lambda_n)$ is resonant if and only if $(\lambda_{\sigma(1)},...,\lambda_{\sigma(n)})$ is resonant for every permutation of the set of indexes $\{1,2,...,n\}.$ Keeping this in mind, we say that a singular vector field $X\in \vf{}{n}$ is resonant if the the $n$-tuple formed by the eigenvalues of the  linearization matrix $D_0X$ of $X$ is resonant.
Finally, we say that a singular vector field $X\in \mfx(\C^n,0)$ is in Poincar\'{e} domain if $0$ does not belong to the convex hull of  $\spec(D_0X)=\{\lambda_1,...,\lambda_n\},$ that is, there is no choice of $t_j \in [0,1]$ for $1 \leq j \leq n$ such that 
\[t_1+\cdots +t_n=1\  \ \ \mbox{ and }\ \ \ t_1\lambda_1+\cdots+t_n\lambda_n=0.\]

\begin{P-D*}[cf. \cite{ilyashenko-yakovenko2008} p. 62] \label{teo.forma.normal.de.P-D} Let $X\in\mfx(\C^n,0)$ be a singular vector field in the Poincar\'{e} domain. Then $X$ has only finitely many resonances and it is analytically conjugated to 
	\[Ax+\sum c_{k,m}F_{k,m},\]
	where $A$ is in Jordan normal form, $c_{k,m}\in \C$ and the $F_{k,m}$ are the resonant monomials of $X.$ In particular, if $X$ is in the Poincar\'{e} domain and it has 
	no resonances, then $X$ is analytically linearizable.
\end{P-D*}
\strut 

\subsection{Stable manifold theorem}
We denote by $(W,p)$ the germ of an analytic variety at a point $p$. 
A germ of analytic variety  $W\subset (\C^n,0)$ is said to be invariant by $F\in\diff{}{n}$ if the germs of $W$ and $F(W)$ coincide at the origin. Additionally, we say that $W$ is stable if there exists a neighborhood $U$ of $0$ where $F$ and $W$ are defined and satisfy
\begin{enumerate}
	\item $F(U\cap W)\subset U\cap W$ and
	\item for each $x\in U\cap W,$ the positive orbit $(F^m(x))_{m\in \N}$ converges to zero. 
\end{enumerate}
Analogously, a variety $W\subset (\C^n,0)$ is invariant by $X \in \mfx (\C^n,0)$ 
if $X(x)$ is tangent to $W$ at $x,$ for every regular point $x\in (W,0).$ In particular, $W$ is invariant by $\exp(X).$ We say that $W$ is stable by $X$ if is stable for its real flow, i.e. 
for any $x \in U \cap W$, we have $\mathrm{exp}(t X)(x) \in U \cap W$ for any 
$t \in {\mathbb R}^{+}$, and $\lim_{t \to \infty} \mathrm{exp} (t X)(x) =0$.

\begin{hsmtb*}[cf. {\cite[p. 26]{ruelle1989}}, {\cite[p.107]{ilyashenko-yakovenko2008}}]
\label{teo.tve.for.bih.} \label{thm_hsmt}
	Let $F\in\diff{}{n}$ and $\rho\in (0,1].$  Let 
\[ L^{-} =\oplus_{\lambda \in A_{\rho}^{-}} \ker (D_0 F - \lambda \mathrm{id})^{n} \] 
be the sum of the generalized eigenspaces associated to the eigenvalues of $D_0F$ in
	\[A_\rho^-=\{\lambda\in \spec (D_0F);\ |\lambda| < \rho\}.\] 
	Then there exists a unique $F$-stable manifold $W^{-}$   
	whose tangent space at $0$ is $L^{-}$.
	\end{hsmtb*}
\begin{cor}\label{cor_fin.orb_imp_mod1}
Let $F\in \diff{<\infty}{n}$. 
Then $|\lambda|=1$ for any $\lambda\in \spec(D_0F).$
\end{cor}

\begin{proof} Suppose $F$ has a eigenvalue $\lambda$ such that $|\lambda|\neq 1.$ So, up to change $F$ by $F^{-1},$ we can suppose $|\lambda|<1.$ Thus $F$ admits a stable manifold $W^{-} \neq \{0\}$ invariant by $F$ and associated to $A_1^-.$ Hence, the points of $W^{-}$ close to $0$ have infinite orbits.\end{proof}

\begin{hsmtv*}[cf. { \cite{carrillo2014}}]
\label{thm_hsmtv}
	Let $X\in \mfx (\C^n,0)$ be a singular vector field and $\theta \in {\mathbb R}_{<0}$. Suppose  
	\[S_{\theta}^{-}:=\{\lambda\in \spec(D_0X);\ \mathrm{Re}(\lambda) \leq \theta\}\neq \emptyset\]
	and denote by  $L_{\theta}^{-} =\oplus_{\lambda \in S_{\theta}^{-}(X)} \ker (D_0 X - \lambda \mathrm{id})^{n}$  the direct sum of the generalized eigenspaces associated to the eigenvalues in $S_{\theta}^{-} (X).$  Then $X$ admits a unique germ of stable manifold $W_{\theta}^{-}$ whose tangent space at $0$ is $L_{\theta}^{-}$. 
\end{hsmtv*}


\section{Finite orbits and one parameter groups}
\label{sec:one_parameter}
In this section we show Theorem \ref{thm_flow_fin_orb}.
\begin{proof}
	Suppose  that $X$ has the finite orbits property. Applying the stable manifold theorem for $X$ we conclude that $\spec (D_0X)\subset i\R.$ The proof now follows by induction on $n.$
	
	If $n=1,$ then the finite orbits property is equivalent to periodicity and hence  
	$\spec (D_0X)=\{\lambda\}\subset 2\pi i\Q.$
	
	Now, suppose $n\geq 2$ and assume the theorem holds in dimension less than $n.$ 
	Set 
	\[A_{>0}=\{\lambda\in \spec(D_0X); \lambda\in i\R_{>0}\}\ \mbox{ and }\ A_{<0}=\{\lambda\in \spec(D_0X); \lambda\in i\R_{<0}\}.\]
	Suppose $A_{>0}\neq\emptyset$. Let us prove that $A_{>0}\subset 2\pi i \Q.$ 

Setting $\tilde{X}=iX,$ we obtain a vector field whose complex trajectories coincide with those of $X,$ both interpreted as sets. 
	The subset of $D_0 \tilde{X}$ consisting of eigenvalues with negative real part is $i A_{>0}$. Therefore, $\tilde{X}$ 
	admits a stable manifold $V_0$, which is invariant by $X$, such that $\mathrm{Spec} (D_0 X_0)= A_{>0},$ where $X_0= X_{|V_0}$. Since $X$ has the finite orbits property so does $X_0.$ 
	If $A_{>0}\neq \spec(D_0X),$ then $\dim V_0<n$ and so by induction hypothesis we have $A_{>0}\subset 2\pi i \Q.$
	Therefore, we can suppose $A_{>0}=\spec(D_0X)$ and $V_0$ is an open subset of $\C^n$. Hence $X=X_0$ 
	is in the Poincar\'{e} domain. If $X$ has no   resonances then $X$ is analytically linearizable and we have $A_{>0}\subset 2\pi i\Q$ by Corollary \ref{cor.of.sse.r.campo}. Suppose, then, that $X$ 
	admits resonances. It follows that $k:= \sharp   (\spec(D_0 X_0))$ is greater than
	$1$.   Set $\spec( D_0 X_0)=\{\lambda_1,...,\lambda_k\}$ with $\im(\lambda_1)>\cdots>\im(\lambda_k).$
	By applying the  Stable Manifold Theorem  to $i X$ and 
$\theta = i \lambda_1$, $\theta = i \lambda_2, \hdots, \theta = i \lambda_k$, 
we find invariant manifolds $V_1, V_2, \hdots, V_{k-1}, V_k = V_0$ such that 
$\spec (D_0 X_j) = \{\lambda_1, \hdots, \lambda_j \}$ where $X_j : =  X|_{V_j}$ 
for all $1\leq j\leq  k.$
Consequently, $X_1$  has no resonances and so is analytically linearizable. It follows from Corollary \ref{cor.of.sse.r.campo} that $\lambda_1\in 2\pi i\Q.$

Suppose that $\lambda_1,...,\lambda_l\in 2\pi i \Q,\ l<k.$ We are going to show that 
$\lambda_{l+1}\in 2\pi i \Q$. Since $\im(\lambda_1)>\cdots>\im(\lambda_k),$ it follows that $X_{l+1}$ is linearizable (in which case $\lambda_1,....,\lambda_{l+1}\in 2\pi i \Q$) or the 
possible resonances of $X_{l+1}$ have the form
	\[\lambda_j=|M_{j+1}|\lambda_{j+1}+\cdots+|M_{l+1}|\lambda_{l+1}, \ \ \sum_{k=j+1}^{l+1}|M_k|\geq 2,\]
	where $M_j\in \Z_{\geq 0}^{n_j}$, and $n_j$ is the algebraic multiplicity of $\lambda_j$ for $j=1,...,l.$ By using the 
	Poincar\'{e}-Dulac normal form, we have coordinates $x=(x_1,...,x_{l+1})\in \C^{n_1}\times\cdots \times\C^{n_{l+1}}=\C^m$ 
	such that $X_{l+1}\sim Ax+(*,*,...,*,0),$ where the first spot corresponds to the $n_1$ first coordinates, the second spot to the next $n_2$ coordinates and so on. If there exists some resonance with $M_{l+1}\neq 0,$ then we can write $\lambda_{l+1}$ as a combination of $\lambda_1,...,\lambda_l$ with rationals coefficients, that is,  $\lambda_{l+1}\in 2\pi i\Q.$ If  all resonances satisfies $M_{l+1}=0,$ then $X_{l+1} \sim Ax+(*,*,...*,0,0)$ 
and we see that the manifold
$W=\{x_l=0\}\subset \C^m$ has dimension $n_1+n_2+...+n_{l-1}+n_{l+1},$ is invariant by $X_{l+1}$ and the  restriction $X|_W$ has eigenvalues $\lambda_1,...,\lambda_{l-1},\lambda_{l+1}.$ By hypothesis of induction, we see that $\lambda_1,...,\lambda_{l-1},\lambda_{l+1}\in 2\pi i\Q.$ This shows that $A_{>0}\subset 2\pi i \Q.$

Analogously, $A_{<0}$ is contained in $2\pi i \Q$ and hence $\spec(D_0X)\subset 2\pi i \Q$.
\end{proof}
\strut

The inverse proposition of the above theorem is not true, even in dimension one, 
as we can see in the next example.

\begin{example}
	\emph{The vector field $X(z)=z^2 \pz$ has spectrum $\spec(D_0X)=\{0\},$ but it does not have the finite orbits property, since $\exp(X)=z+z^2+O(z^3)$ is clearly non-periodic.}
\end{example}

\section{Fixed Point Curve Theorem}
\label{sec:fixed}
In this section we prove that if $F\in\diff{<\infty}{2}$ then 
some iterated $F^m$ admits a curve of fixed points at $0.$ 
First, we use constructions and ideas featuring in Mattei and Moussu \cite{mattei-moussu1980}, Rebelo and Reis  \cite{rebelo-reis2015a} and Pérez-Marco \cite{perez-marco1997} to show that there is a non-trivial continuum $K$ containing the origin and 
satisfying $F(K)=K$. 
The set $K$ consists of periodic points of $F$ and can be obtained as a limit of compact sets where we consider the Hausdorff topology on the  compact  subsets of $\overline{B}$. 
Finally, we will use the theory of semianalytic sets (see \cite{lojasiewcz1964}  \cite{bierstonemilman1988}) to show that the continuum $K$ is contained in an analytic curve which is invariant by some iterate of $F.$

\subsection{Continua}
For the sake of simplicity, we recall in this section
the Sierpi\'nski theorem and the Hausdorff topology on compact sets.

A topological space $X$ is called a \textit{continuum} if $X$ is both connected and compact. 
The next result will be a key ingredient in the description of the
connected components of the set of periodic points of a finite orbits local biholomorphism.
\begin{sierpinski*}[see {\cite[p.358]{engelking1989}}]
	Let $X$ be a continuum that has a countable cover $\{X_j\}_{j=1}^{\infty}$ by pairwise disjoint closed subsets. Then at most one of the sets $X_j$ is non-empty.
\end{sierpinski*}

Now, we define the Hausdorff topology.
Let $(M,d)$ be a metric space and denote by $H(M)$ the space of bounded, non-empty closed subsets  of $M.$  Note that $H(M)$ is the set of  compact subsets of $M$ if $M$ is compact. 
We define the Hausdorff metric $\rho:H(M)\times H(M)\to[0,\infty)$ by
\[\rho(A,B)=\max\{\sup_{x\in A}d(x,B),\ \sup_{y\in B}d(A,y)\}.\]
Consider $A\subset M$ and $\varepsilon>0$. We
define the $\varepsilon$-neighborhood of $A$ by
\[V_\varepsilon(A)=\bigcup_{x\in A}B_\varepsilon(x), \]
where $B_{\epsilon} (x)= \{y \in M; d(y,x) < \epsilon\}$.
The Hausdorff  metric satisfies
\[\rho(A,B)=\inf\{\varepsilon>0;\ V_\varepsilon(A)\supset B\ \mbox{and}\ V_\varepsilon(B)\supset A\}\]
for all $A,B \in H(M)$.
Moreover, the metric space  $(H(M),\rho)$ is compact if $M$ is compact
(cf. \cite[p. 280]{munkres2000}, \cite{henrikson1999}).
Note that a compact subset $K$ of $M$ is connected if and only if any pair of points $p,q \in K$
can be joined by an $\epsilon$-chain of points in $K$ 
for any $\epsilon >0$ \cite[p. 81]{Newman1992}. 
This criterium can be used to show that the subset  of $H(M)$ consisting of continua is 
compact if $M$ is compact.

\strut


\subsection{Invariant curves}
 A formal curve at $0\in \C^2$ is a proper radical ideal $\Gamma=(f)$ of $\C[[x, y]].$
Such a condition is equivalent to $f$ being reduced, i.e. $f$ has no multiple irreducible factors. 
\begin{enumerate}
	\item[(i)] We say that $\Gamma$ is invariant by $F\in \fdiff{}{2}$ if  $\Gamma\circ F=\Gamma$, i.e., $f$ divides $f\circ F.$
	\item[(ii)] We say that $\Gamma$ is a fixed point curve of $F\in \fdiff{}{2}$ if  $f$ divides $x\circ(F-\id)$ and $y\circ(F-\id)$.
	\item[(iii)] We say that $\Gamma$ is invariant by $X\in \fvf{}{2}$ if $X(\Gamma)\subset \Gamma,$ i.e., $f$ divides $X(f).$ 
	\item[(iv)] We say that $\Gamma$ is a singular curve of $X$ if $f$ divides $X$.
\end{enumerate}
In the case where the curve $\Gamma=(f)$ is a radical ideal of $\C \{ x, y \}$,
we identify it with the germ of analytic set 
$V_\Gamma=(f=0)$ and the conditions above coincide with the natural ones
for $F \in \diff{}{2}$ and $X \in \vf{}{2}$:
\begin{enumerate}
	\item[(i)] We have the equality $F(V_\Gamma)=V_\Gamma$ 
	of germs of analytic sets at $0.$
	\item[(ii)] $F|_{V_\Gamma}=\id.$ 
	\item[(iii)] $X$ is tangent to $V_\Gamma$ at any of its regular points.
	\item[(iv)] $X|_{V_\Gamma}=0.$ 
\end{enumerate}


\begin{lemma}[cf. \cite{ribon2005}]\label{lema.curva.fixa.sing.}
	Let $X\in \fvf{N}{2}$ and let $\Gamma$ be a formal curve at $0.$
	\begin{enumerate}
		\item[(a)] $\Gamma$ is invariant by $X$ iff $\Gamma$ is invariant by $\exp(X).$
		\item[(b)] $\Gamma$ is singular curve of $X$ iff $\Gamma$ is a fixed point curve of 
		$\exp(X)$.
\end{enumerate} \end{lemma}
\strut

\begin{lemma}\label{lema.da.unip.com.cur.inv}
	Let $V,W\subset (\C^2,0)$ be different germs of non-trivial analytic sets. 
	Consider $\psi\in \diff{}{2}$ such that  $V,W\subset \mathrm{Fix}(\psi)$.
	Then  $\psi$ is tangent to identity.	
\end{lemma}
\begin{proof}
	We have $V\neq \{0\}\neq W$ and  $V \neq W$ by hypothesis.
	If $V=(\C^2,0)$ or $W=(\C^2,0)$ then $\psi=id.$ Therefore, we can suppose that $V$ and $W$ are analytic curves with reduced  equation $f=0$ and $g=0,$ respectively, at $0$. 
	So both $f$ and $g$ divide $x\circ \psi-x$ and $y\circ \psi-y.$ If $\ordem(f)\geq 2$ or $\ordem(g)\geq 2,$ then the first jet $J^1\psi$ of $\psi$ is equal to $Id.$ So we can assume
	$\ordem(f)=\ordem(g)=1$. Since $V \neq W$, we deduce
	$fg|x\circ\psi-x$ and $fg|y\circ\psi-y.$ Again, we obtain $J^1\psi=Id.$ This concludes the proof.
\end{proof}

\subsection{Connected components of the set of periodic points}

Let us assume that $F\in\diff{<\infty}{2},$ $U$ is a neighborhood of $0$ in which $F$ is defined and injective and fix  a closed ball $\overline{B}_{\rho} (0)$ such that 
	$ \overline{B}_{\rho} (0) \subset U.$ Let us also set
\begin{equation}
\label{equ:perk}
  \mathrm{Per}_k (F) 
:=\{p\in  {B}_{\rho} (0);\ F(p),...,F^{k-1}(p)\in   {B}_{\rho} (0) \ \mbox{and}\ F^{k}(p)=p\} 
\end{equation}
and
\[  \overline{\mathrm{Per}}_k (F)
	:=\{p\in  \overline{B}_{\rho} (0);\ F(p),...,F^{k-1}(p)\in   \overline{B}_{\rho} (0) \ \mbox{and}\ F^{k}(p)=p\}    \]
for $k \in {\mathbb N}$ and 
\begin{equation}
\label{equ:per}
   \mathrm{Per}  (F)=~\bigcup_{k\in \N}   \mathrm{Per}_k (F),  \ \ 
	\overline{\mathrm{Per}}  (F)=~\bigcup_{k\in \N}   \overline{\mathrm{Per}}_k (F)  .  
\end{equation}

\begin{lemma}\label{lema.do.compacto.n=2} \label{lemma_exist.of.comp.}
	There is a subset $K \subset  \overline{\mathrm{Per}}  (F)$ with the following properties.
	\begin{enumerate}
		\item $K$ is a continuum;
		\item $0\in K$ and $K\cap \partial    {B}_{\rho} (0)  \neq \emptyset;$
		\item $F(K)=K$
	\end{enumerate}
\end{lemma}

\begin{proof}
	We have $\spec (D_0F)\subset S^1$ as a consequence of   the Stable Manifold Theorem.  For each $n\in \N,$ set $F_n=e^{-1/n}F.$  Thus,  $F_n$ is a biholomorphism with $\spec (D_0F_n)\subset \{\lambda\in \C;\ |\lambda|<1\}.$
	
	\strut
	
	\textbf{Claim.} 
	 We denote $B_r = B_r (0)$ and $\overline{B}_r = \overline{B}_r (0)$ for $r \in \R^{+}$
		and $B= B_{\rho}$, $\overline{B}= \overline{B}_{\rho}$. 
	There are sequences $(k_n)$ of positive integer numbers and $(r_n)$ of positive real numbers 
	such that:
	\begin{enumerate}
		\item $\lim_{n \to \infty} r_n=0$ and the closed ball 
		$\overline{B}_{r_n}$ satisfies 
		\[\cup_{j=1}^{\infty} F_{n}^{j}(\overline{B}_{r_n}) \subset \overline{B};\]
		\item $\overline{B}_{r_n},F_n^{-1}(\overline{B}_{r_n}),...,\ F_n^{-k_n}(\overline{B}_{r_n})\subset \overline{B}$ and $F_n^{-k_n}(\overline{B}_{r_n})\cap \partial B \neq \emptyset.$
	\end{enumerate}
	
	Let us assume this for a moment to prove the lemma. We define 
	\[V_n = \overline{\cup_{j \geq -k_n} F_{n}^{j}(\overline{B}_{r_n})}.\]
	Then $V_n$ is connected since it is the closure of a union of connected sets that have the origin as a common point. Thus, $V_n$ is a continuum contained in $\overline{B}$ such that 
	$F_{n}^{j} (V_n) \subset V_n$ for all $j \geq 0$ and there exists $p_n\in V_n\cap\partial B.$
	Passing to a subsequence if necessary, we can assume that $V_n\to K,$ in the Hausdorff topology of compact subsets of $\overline{B},$ and also $p_n\to p\in K\cap \partial B.$ In particular, $K$ is a continuum containing the origin such that $K\cap \partial B\neq \emptyset$. Since $F_{n}(V_n) \subset V_n$ for every $n \in {\mathbb N},$ $(F_n)_{n \geq 1}$ converges to $F$ uniformly in $\overline{B}$ and $V_n \to K,$ we deduce that  $F(K) \subset K.$ Therefore, since $F$ has the finite orbits property, we obtain $K \subset  \overline{\mathrm{Per}}  (F)$.
	In particular, $K$ is contained in the image of $F|_K$ and  hence $F(K)=K$.
	
	\strut
	
	\textbf{Proof of the claim.} 
	
	Let us construct $0 < r_n < 1/n$ and $k_n$. Since the origin is an attractor for $F_n$,
	there exists $R \in (0, 1/n)$ such that the closed ball $\overline{B}_{R}$ is contained in the
	basin of attraction of $0$ and satisfies 
	$\cup_{j=1}^{\infty} F_n^{j}(\overline{B}_{R}) \subset \overline{B}$.
	We claim that there exists  $k_n \in {\mathbb N}$ such that
	\[ F_n^{-1} (\overline{B}_{R}) \cup \hdots \cup F_n^{- (k_{n}-1)} (\overline{B}_{R}) 
	\subset \overline{B}  \ \ \mathrm{and} \ \ 
	F_n^{-k_n} (\overline{B}_{R}) \setminus \overline{B} \neq \emptyset . \]
	Assume, aiming at contradiction, that no such $k_n$ exists. 
	Denote $A = D_0 F_n^{-1}$ and $A^{k} = (a_{ij;k})_{1 \leq i,j \leq 2}$ for $k \in {\mathbb Z}$. 
	We have 
	\[ a_{11;k} = \frac{1}{(2 \pi i)^{2}} \int_{|x| = |y| = \frac{R}{2}} \frac{x \circ F_{n}^{-k}}{x^2 y} dx dy 
	\implies |a_{11;k}| \leq \frac{1}{(2 \pi)^{2}} (\pi R)^{2}  \rho {\left(\frac{2}{R} \right)}^{3}  
	= \frac{2 \rho}{R} \]
	for any $k \in {\mathbb N}$. Analogously, we obtain 
	$ |a_{12;k}| \leq 2 \rho / R$, $ |a_{21;k}| \leq 2 \rho / R$ and $ |a_{22;k}| \leq 2 \rho / R$ 
	for any $k \in {\mathbb N}$. We proved that the sequence $(A^{k})_{k \geq 1}$ is bounded, 
	contradicting $\mathrm{spec} (A) \subset \{ z \in {\mathbb C}; |z| >1 \}$.
	
	By defining 
	\[ r_n = \inf \{ s \in (0, R);  F_n^{-k_n}  (\overline{B}_s) \setminus \overline{B} \neq \emptyset \}, \]
	we obtain $k_n$ and $r_n$ satisfying the desired properties.
	\end{proof}
\strut

\begin{defi}
	\emph{Let $M$ be a real analytic manifold. A subset $X$ of $M$ is semianalytic if each $p\in M$ has a neighborhood $V$ such that $X\cap V$ has the form
		\[X\cap V=\bigcup_{i=1}^{m}\bigcap_{j=1}^nX_{ij},\]
		where $X_{ij}=\{f_{ij}=0\}$ or $X_{ij}=\{f_{ij}>0\}$ with $f_{ij}$ real analytic on $V.$}
\end{defi}

\begin{rem}
	\emph{Notice that every (real or complex) analytic set $X$ is semianalytic with $m=1$ and $X_{ij}=\{f_{ij}=0\}.$ Moreover, since  $\{f_{ij}<0\}=\{-f_{ij}>0\},$ $\{f_{ij}\geq 0\}=\{f_{ij}=0\}\cup\{f_{ij}>0\}$ and $\{f_{ij}\leq 0\}=\{f_{ij}=0\}\cup\{f_{ij}<0\},$ we can add these types of sets to the possible options for $X_{ij}$ to obtain an alternative definition.}
\end{rem}

\begin{example}\emph{
		For each $r>0$ the closed ball 
		$ \overline{B}_r(0) $ 
		is semi-analytic in $\C^n,$ since it can be written as 
		$ \overline{B}_r(0) =\{f\geq0\},$ where $f(x)~=~r^2-|x_1|^2-\cdots -|x_n|^2.$}
\end{example}

\begin{lemma}\label{lemma_ext_of_the_comp} Under the hypotheses above, 
	each  set $\overline{\mathrm{Per}}_k (F)$  is semianalytic in $U$, has finitely many 
	connected components, 
	and each  of its connected components is semianalytic and path-connected.
\end{lemma}

\begin{proof}
	 We denote ${\mathcal P}_k = \overline{\mathrm{Per}}_k (F)$. 
	Let $p\in  {\mathcal P}_k $. 
	The local biholomorphism $F^k$ is  well-defined  in some neighborhood $V$ of $p.$ 
	Moreover, we have 
	\[   {\mathcal P}_k  \cap V
	=\left(\cap_{j=1}^4\{f_j=0\}\right)\cap\left(\cap_{l=0}^{k-1}\{f\circ F^l\geq 0\}\right),\]
	where $f_1=\re(x\circ F^k-x),\ f_2=\im(x\circ F^k-x),\ f_3=\re(y\circ F^k-y),$ $f_4=\im(y\circ F^k-y),$ and $f(x,y)~=~r^2-|x|^2-|y|^2.$ Therefore,  ${\mathcal P}_k$  is semianalytic. Now, by Corollary 2.7 in \cite{bierstonemilman1988}, we know that each connected component of   ${\mathcal P}_k$  is also semianalytic and the family of connected components of ${\mathcal P}_k$ is locally finite. Since  ${\mathcal P}_k$  is compact, it has finitely many connected components. Finally, by using Theorem 1 in \cite{lojasiewcz1964}, we know that  
	${\mathcal P}_k$  is triangulable and so is locally path connected. Thus, each connected component of ${\mathcal P}_k$  is path connected.
\end{proof}




\begin{lemma}\label{lema.esp.const.} \label{lema.da.pre.estab.da.uniao}
	Let $C$ be a connected component of some $\overline{\mathrm{Per}_{kl}} (F)$ 
	and suppose that
	\[E=\{p\in C;\ F^k(p)=p\  \mbox{ and the germ }F^k_p \mbox{ of } F^{k}
	\mbox{ at } p \mbox{  is unipotent } \}\]
	is non-empty. 
	Then $C$ is a subset of $\overline{\mathrm{Per}_{k}} (F)$ 
	and $D_{p} F^{k}$ is unipotent for all $p \in C$.
\end{lemma}

\begin{proof}
	Denote $\psi = F^{k}.$  We know that if $\psi(p)=p$  then the characteristic polynomial of $D_p\psi$ is \[P_{D_p\psi}(x)=x^2-\mathrm{tr}(D_p\psi)x+\det D_p\psi=x^2-Sx+P,\] where $S$ is the sum and $P$ is the product of the eigenvalues of $D_p\psi.$ Therefore,
	\[E= \{ p \in C ;\ \psi(p)=p, \ \mathrm{tr}(D_p \psi)=2 \  \mathrm{ and }\det(D_ p \psi) = 1 \}.\]
	In order to prove the lemma it suffices to show that $E=C.$ Since $E\neq\emptyset,$ the set $C$ is connected,	 and $E$ is closed in $C,$ it suffices to show that $E$ is open in $C.$ Consider $p\in E .$ Let us first  prove that  
	the germ $(C,p)$ of $C$ at $p$ is contained in $\overline{\mathrm{Per}_{k}} (F)$. 
	Set
	\[A = \{ q \in U; \psi (q) = q \}\ \ \ \mathrm{ and }\ \ \ 
	B=\{q\in U;\ \psi^{l}(q)=q\}.\]
	It is obvious that $A$ and $B$ are  analytic  	
	and $(C,p)\subset(B,p)$. Now since $\psi(p)=p$ and $D_p\psi$ is unipotent, we can consider the infinitesimal generator $X_p$ of the germ $\psi_p,$ i.e., the nilpotent formal vector field  $X_p$ such that $\exp(X_p)=\psi_p.$ Consequently, $lX_p$ is the infinitesimal generator of $\psi^l_p=F^{kl}_p.$  Since  $(\psi_p^l)|_{B}=\id,$   Lemma \ref{lema.curva.fixa.sing.},  
	applied  to   the germ of $B$ at $p$,  implies that $(B,p)\subset \Sing(lX_p)=\Sing(X_p)$ and therefore  $\psi_{p}|_{B}=\id$.  
	In particular, we obtain  $ (C,p)\subset \overline{\mathrm{Per}_{k}} (F) $. 
	
	Now let us prove that $(C,p)\subset (E,p).$ Define $f: A \to {\mathbb C}$ by setting
	$f(q)= \det (D_q \psi)$. The restrictions of $f$ to the 
	irreducible components of the germ of $A$ at $p$ are
	holomorphic functions. As $F$ has the finite orbits property, it follows by the Stable Manifold Theorem that the eigenvalues of $D_q \psi$ have modulus $1$ and then 
	the image of $f$ is contained in the circle  $S^1$, 
	since  $\det (D_q \psi)$ 
	is the product of the eigenvalues of $D_q \psi.$ 
	In particular,  the image of $f$  does not contain any open set. 
	 We obtain that $f$ is locally constant in the neighborhood of $C \cap A$ in $A$
		by the open mapping theorem. 
		Since $f(p)=1$ and $ (C,p)\subset A$, we deduce that 
		$(C,p) \subset  \{ f = 1 \}$.
	Now we consider a function $g: A \to {\mathbb C},$ defined by $g(q)= \mathrm{tr} (D_q \psi).$ For each $q$ in some neighborhood of $p$ in $A$,
	 the eigenvalues $\lambda$ and $\mu$ of $D_q \psi$ satisfy $|\lambda|=|\mu|=1$ and 
		$\lambda \mu =1$ since $\spec (D_p \psi) \subset S^1$ and   $(C,p) \subset  \{ f = 1 \}$. 
		Therefore, we obtain
		$ \mathrm{tr} (D_q \psi) = \lambda + \overline{\lambda} = 2 \mathrm{Re}(\lambda) \in [-2,2]$. 
		Again, the restriction of $g$ to the irreducible components of $(A,p)$ defines holomorphic 
		functions whose images do not contain any open set. 
		Hence, $g \equiv g(p)=2$ is constant in a neighborhood of $p$ in $A$. 
		It follows that  $(C,p)\subset (E,p).$ This concludes the proof. 
\end{proof}
\strut

We have already seen in Lemma \ref{lemma_ext_of_the_comp} that each 
 $\overline{\mathrm{Per}}_k (F)$  has finitely many connected components, say, 
 $C^{k}_1,....,C^{k}_l$. 
Now, let us consider the \textit{family of all components of $F$ in $D,$} namely, 
\[{\mathcal A}:=\{C_j^k;\ k\in \N\ \mathrm{ and }\ C_j^k\ \mbox{ is a connected component of }  \overline{\mathrm{Per}}_k (F)   \ \}.\] 
We say that two components 	$C,D\in \mathcal A$ are equivalent 
(and then we write $C\sim D$) if there are  
\[C_{\alpha_1}^{k_1},...,C_{\alpha_r}^{k_r}\in\mathcal A\]
such that $C=C_{\alpha_1}^{k_1},\ D=C_{\alpha_r}^{k_r}$ and  $C_{\alpha_s}^{k_s}\cap C_{\alpha_{s+1}}^{k_{s+1}}\neq \emptyset$ for all $1\leq s<r.$ Notice that $\sim $ is a relation of equivalence in  $\mathcal A.$ 

\begin{rem}\emph{
		If $C_{\alpha_1}^{k_1},...,C_{\alpha_r}^{k_r}$ are components in $\mathcal{A}$ such that $C_{\alpha_s}^{k_s}\cap C_{\alpha_{s+1}}^{k_{s+1}}\neq \emptyset$ for all $1\leq s<r,$ then there are $j,k\in \N$ such that $\bigcup_{s} C_{\alpha_s}^{k_s}\subset C_{j}^{k}.$ If fact, since the union is connected and is contained in  $\overline{\mathrm{Per}}_{k_1 \hdots k_r} (F)$,  
		it is contained in $C^{k_1...k_r}_{j}$ for some  $j.$} 
\end{rem}

\begin{lemma}[stability of classes] \label{teo_stab_of_cla}
	If $[C]$ is an equivalence class (possibly infinite) in ${\mathcal A}/\sim,$ then 
	\[\bigcup_{C_j\in [C]}C_j= C^{k_0}_{j_0}\]
	for some $j_0,k_0.$
\end{lemma}

\begin{proof}
	The result is obvious if there is at most one non-unitary component in $[C]$,
 i.e. a component $C_{j}^{k}$ of $[C]$ such that $\sharp C_{j}^{k} >1$. 
	Thus, we can assume that there are two distinct  non-unitary  
	components $C_a$ and $C_b$ of 
	 $[C]$ 
		such that $C_a \cap C_b \neq \emptyset$. 	
		Suppose that $C_a$ is a connected component of ${\mathcal P}_a$ and
		$C_b$ is a connected component of ${\mathcal P}_b$, where we denote 
		${\mathcal P}_j =\overline{\mathrm{Per}}_j (F)$. 
	
	\strut
	
	Let us show that there is $p\in C_a\cap C_b$ such that $D_pF^{ab}$ is unipotent. 
	First  of  all, there is $p \in C_a \cap C_b$ such that $(C_a,p) \neq (C_b,p)$ in 
	 $\overline{B}$ since $C_a$ and $C_b$ are connected. 
	Let $(V_a,p)$ and $(V_b,p)$ be the 
	gems of analytic set  of equation $F^a = Id$ and $F^b=Id,$ respectively, defined in some neighborhood of $p$ in $U.$ Since $C_a$ and $C_b$ are  non-unitary, 
	we have $\dim (V_a,p) \geq 1$ and $\dim (V_b,p) \geq 1.$

	We claim that $(V_a , p) \neq (V_b , p)$. Otherwise, we have 
	$(C_a\cup C_b,p)\subset   {\mathcal P}_a \cap {\mathcal P}_b $ and hence 
	\[  (C_a,p)=(C_a\cup C_b,p)=(C_b,p),\] 
	that contradicts the choice of $p$. Therefore we obtain $D_pF^{ab}=\mbox{id}$
	by Lemma \ref{lema.da.unip.com.cur.inv}.

	Consider the  connected component $C_{\ell}^{ab}$ of 
		${\mathcal P}_{ab}$  containing $p$. 
	Let $A$ be a connected union of finitely many components of $[C]$ that contains $p$.
	Then $A \subset  C_{\ell}^{ab}$ by Lemma 4.5. By varying $A$, we deduce  
	$\cup_{C_j \in [C]} C_j = C_{\ell}^{ab}$.
\end{proof}

\subsection{Structure of the set of periodic points}
Now, we combine the previous results to show Theorem \ref{thm:structure} and the Fixed
	Point Curve Theorem. 
\begin{proof}[Proof of Theorem \ref{thm:structure} and the Fixed Point Curve Theorem]
	Let $B$ an open ball such that $F$ and $F^{-1}$ are defined in a neighborhood of 
	$\overline{B}$.   
	Let  ${\mathcal P}$ be a connected component of  $\overline{\mathrm{Per}} (F)$. 
		It  is a countable union of elements of the family  $\mathcal A$ of components of $F$ in
	 $\overline{B}$. Up to consider only maximal components of $F$ in $\overline{B}$,
		we can suppose that such a union is disjoint by Lemma \ref{teo_stab_of_cla}. 
	 Therefore, we obtain ${\mathcal P} = C$ for some $C\in \mathcal A$ by 
		Sierpinski Theorem. Hence ${\mathcal P}$ is a connected component of some 
		$\overline{\mathrm{Per}}_k (F)$. Both $C$ and $\overline{\mathrm{Per}}_k (F)$ 
		are semianalytic by Lemma \ref{lemma_ext_of_the_comp}.

		Let  ${\mathcal Q}$ be a connected component of $\mathrm{Per} (F)$. 
		It is contained in a connected component ${\mathcal Q}'$ of $\overline{\mathrm{Per}} (F)$. 
		Then ${\mathcal Q}'$ is a semianalytic subset of  $\overline{\mathrm{Per}}_k (F)$
		for some $k \in \N$ by the first part of the proof. 
		We obtain ${\mathcal Q} \subset {\mathrm{Per}}_k (F)$ and hence
		the set ${\mathcal Q}$ is given locally by the equation $F^{k} = \mathrm{id}$. It follows that
		${\mathcal Q}$ is a complex analytic subset of the open ball. 
		Since the set  ${\mathrm{Per}}_k (F) \cap {\mathcal Q}'$ is semianalytic and relatively compact,
		it follows that it has finitely many connected components \cite[Cor. 2.7]{bierstonemilman1988}
		that are all semianalytic. We deduce that  ${\mathcal Q}$ is a
		semianalytic subset of ${\mathbb C}^{2}$.

		We claim that $\dim ({\mathcal Q},p)  \geq 1$ for any $p \in {\mathcal Q}$. 
		This is equivalent to the property $\sharp {\mathcal Q} >1$ since ${\mathcal Q}$ is a connected
		component of $\mathrm{Per}_{k}(F)$. 
		First, suppose $0 \in {\mathcal Q}$. Since 
		${\mathcal Q}' \subset \overline{\mathrm{Per}}_k (F)$,  
		there exists a neighborhood $W$ of the origin 
		such that 
		\[ W \cap {\mathcal Q} = W \cap {\mathcal Q}' = W \cap \mathrm{Fix}(F^{k}). \] 
		Note that ${\mathcal Q}'$ is a continuum that contains the non-trivial subcontinuum
		$K$ obtained in Lemma  \ref{lema.do.compacto.n=2}. 
		Since ${\mathcal Q}'$ is a non-trivial continuum, we deduce that the germ of 
		$\mathrm{Fix}(F^{k})$ and then of ${\mathcal Q}$ at the origin have positive dimension
		and thus contains an analytic curve $\Gamma$ passing through $0$. 
		Finally, consider a general connected component ${\mathcal Q}$ of $\mathrm{Per} (F)$. 
		Given $p \in {\mathcal Q}$, 
		there exists a germ of analytic curve $\Gamma'$ at $p$ contained in $\mathrm{Per}_{l}(F)$
		for some $l \in \N$ by the previous discussion.
		Since $(\Gamma', p) \subset {\mathcal Q}$, we obtain 
		$\dim ({\mathcal Q},p)  \geq 1$.
\end{proof}
\strut


\begin{proof}[Proof of Theorem \ref{thm_at_least_one_root}]
 The eigenvalues of $D_0 F$ belong to the unit circle by Corollary \ref{cor_fin.orb_imp_mod1}. 
	By the Fixed Point Curve Theorem, 
	there exists an analytic curve $\Gamma \subset \mathrm{Fix}(F^k)$ 
	through the origin 
	for some $k \geq 1$. Suppose $1 \not \in \mathrm{spec} (D_0 F^k).$ Then $D_0 (F^k - \mathrm{Id})$ is a regular matrix, therefore $F^{k}-\mathrm{Id}$ is a local diffeomorphism at $0.$ 
	But this contradicts $\Gamma \subset  (F^k - \mathrm{Id})^{-1} (0)$. Therefore, 
	we obtain $1 \in \mathrm{spec} (D_0 F^{k})$. Since the eigenvalues of $D_0 F^k$ 
	are $k$th-powers  of eigenvalues of $D_0 F,$ it follows that there exists 
	$\lambda \in \mathrm{spec}(D_0 F)$ such that $\lambda^{k}=1.$	
\end{proof}
\begin{rem}\emph{
 Corollary \ref{cor_orbfin_imp_alg.crit} provides negative algebraic criteria for the finite orbits 
property. For instance, let 
$F(x,y)= (x+ f_1(x,y), y+ f_2 (x,y)) \in  \mathrm{Diff}_{1} ({\mathbb C}^{2},0)$
where $f_j (x,y) = \sum_{k=m}^{\infty} P_{k,j}(x,y)$ is the expansion of 
$f_j \in {\mathbb C} \{x,y\}$ as a sum of homogeneous polynomials 
for $j \in \{1,2\}$, where $m \geq 2$. 
Assume that $P_{m,1}$ and $P_{m,2}$ are relatively prime. 
Since 
\[ x \circ F^{k} - x = k P_{m,1} + h.o.t. \ \mathrm{and} \   y \circ F^{k} - y = k P_{m,2} + h.o.t. \]
we deduce that the fixed point $(0,0)$ of $F^{k}$ is isolated for any $k \in {\mathbb N}$
and hence $F$ does not satisfy the 
finite orbits property.}
\end{rem}
 Later on, we will see that there exists $F \in \diff{}{2}$ with finite orbits but $D_0 F$ has no finite
orbits (see Theorem  \ref{thm_irrat_bih_with_fin_orb}).  It makes sense to study whether the
finite orbits property for other actions naturally associated to $F$ implies 
$\mathrm{spec} (D_0 F)  \subset e^{2 \pi i {\mathbb Q}}$.
We are going to consider the blow-up 
$\pi: \tilde{\mathbb C}^{2} \to {\mathbb C}^{2}$ of the origin and the diffeomorphism 
$\tilde{F}$ induced by $F$ in a neighborhood of the divisor $D:= \pi^{-1}(0)$ 
(see \cite{ribon2005}). 
\begin{cor}
Let $F \in  \diff{}{2}$. Assume that the germ of $\tilde{F}$ defined in the neighborhood of 
$D$ in $\tilde{\mathbb C}^{2}$ has finite orbits. Then $\mathrm{spec} (D_0 F)$ consists of roots of
unity. 
\end{cor}

	\begin{proof}
The diffeomorphism $\tilde{F}_{|D}$ has finite orbits and hence $D_0 F$ induces an element of
finite order of $\mathrm{PGL} (2, {\mathbb C})$.
Thus $D_0 F$ is diagonalizable and has eigenvalues $\lambda, \mu \in \C^{*}$ such that 
$\lambda / \mu$ is a root of unity.  Since at least one eigenvalue of $D_0 F$ is a root of unity
by Theorem \ref{thm_at_least_one_root}, 
we deduce $\mathrm{spec} (D_0 F)  \subset e^{2 \pi i {\mathbb Q}}$.
\end{proof}


\begin{rem}
	\emph{In general, a germ of biholomorphism $H$ does not admit germs of fixed point curves, even when $F\in \diff{<\infty}{2}.$ For example, the germ given by $F(x,y)=(-x,-x-y)$ has the finite orbits property, because  it is linear and $\spec(D_0F)=\{-1\},$ but the only fixed point of $F$ is the origin.  Note, however, that  $F^2(x,y)=(x,2x+y)$ and $\{x=0\}$ is a  fixed point curve of $F^2.$
	 Moreover, the curve $\{x=0\}$ is an irreducible curve invariant by $F$.}
\end{rem}
We conclude this section providing an example of   $F \in \mathrm{Diff} ({\mathbb C}^{2},0)$
that has the finite orbits property but has no irreducible germ of invariant curve. 
The diffeomorphism $F$ is of the form $F= S \circ T$ where 
$S(x,y)= (iy, ix)$, $T  = \mathrm{exp} (X)$ and 
\[ X = xy \left( x \frac{\partial}{\partial x} -   y \frac{\partial}{\partial y} \right) + 
i x^{2} y^{2} \left( x \frac{\partial}{\partial x} +   y \frac{\partial}{\partial y} \right) .\]
Note that $S^{*} X = X$ and hence $S$ and $T$ commute. Moreover $S$ has order $4$. 
Any germ of irreducible curve $\gamma$ that is invariant by $F$ is also invariant by $F^4$
and hence by $\mathrm{exp} (4 X)$.
Since $X$ is the infinitesimal generator of a tangent
to the identity local biholomorphism, we deduce that $\gamma$ is invariant by $X$
 by Lemma \ref{lema.curva.fixa.sing.}. 
Note that the singularity of $X/(xy)$ at the origin is reduced and hence $X/(xy)$ has only two 
irreducible invariant curves by Briot and Bouquet theorem, namely the $x$ and $y$ axes. 
As a consequence, the axes are the unique irreducible germs of $X$-invariant curves.
Since $S$ permutes the axes, it follows that $F$ has no irreducible germ of invariant curve. 

Let us show that $F$ has finite orbits.  It suffices to prove that
$F^{4} = T^{4}$ has finite orbits by  Proposition \ref{propo_equiv_of_fin.orb}.  
Indeed, it suffices to show that
$T$ has finite orbits by the same result.  Next, we study the action induced by $X$ on the leaves of the foliation $d(xy)=0$. Such a 
foliation is preserved by $X$ since 
$X (xy) = 2 i  (x y)^{3} $. 
We can relate the properties of $X$ with those of $Z  = 2 i z^{3} \partial / \partial z$ and 
its time $1$ map $G= \mathrm{exp} (Z)$. 
Indeed, we have 
\[ (xy) \circ T^{k} (x, y) =  G^{k} (x y) \]
for $k \in {\mathbb Z}$.

Fix a small bounded neighborhood $V'$ of $0$ in 
${\mathbb C}$ and a small bounded neighborhood $V$ of $(0,0)$ in 
${\mathbb C}^{2}$ such that $(xy) (V) \subset V'$. 
Consider $(x_0, y_0) \in V$ and denote $z_0 = x_0 y_0$. 
Since the axes consist of fixed points of $T$, we can suppose $x_0 y_0 \neq 0$. 
Assume, aiming at a contradiction that the positive $T$-orbit of $(x_0, y_0)$ in $V$
is infinite. Therefore, $G^{k} (z_0)$ is well-defined and belongs to $V'$ for any $k \geq 0$. 
Since $G$ has a dynamics of flower type, we deduce 
\[ \lim_{k \to \infty} G^{k} (z_0) =0 \ \ \mathrm{and} \ \ 
\lim_{k \to \infty} \frac{G^{k} (z_0)}{|G^{k} (z_0)|} \in \{ e^{\frac{i \pi}{4}}, - e^{\frac{i \pi}{4}} \} .  \]
Assume that the latter limit is equal to $e^{\frac{i \pi}{4}}$. Let us study the variation of 
the monomials $x^{a} y^{b}$ by iteration; we have
\[ x^{a} y^{b} \circ T = x^{a} y^{b}  (1 + (a -b) xy + O(x^{2} y^{2})) \]
for any $(a,b) \in {\mathbb Z}_{\geq 0} \times  {\mathbb Z}_{\geq 0}$.
As a consequence, we get 
\[ |x \circ T^{k+1} (x_0,y_0)| >  |x \circ T^{k} (x_0,y_0)|, \ \ 
|(x^{2} y) \circ T^{k+1} (x_0,y_0)| >  |(x^{2} y) \circ  T^{k} (x_0,y_0)|  \]
for any non-negative integer number $k$. Since $|x|$ increases along the positive $T$-orbit 
of $(x_0,y_0)$ and $\lim_{k \to \infty} (xy)(T^{k}(x_0,y_0)) =0$, we get
$\lim_{k \to \infty} y (T^{k}(x_0,y_0)) =0$. 
Moreover, since $|x^{2} y|$ increases along the positive $T$-orbit of $(x_0, y_0)$ it follows that 
$\lim_{k \to \infty} |x| (T^{k}(x_0,y_0)) =\infty$. This property contradicts that $V$ is bounded. 
The  case $\lim_{k \to \infty} \frac{G^{k} (z_0)}{|G^{k} (z_0)|} =  - e^{\frac{i \pi}{4}}$
is treated in a similar way. Analogously, we can show that the negative $T$-orbit of $(x_0,y_0)$
is finite.  




\section{Non-virtually unipotent biholomorphisms with finite orbits}
\label{sec:non-virtual}
So far, all the examples in the literature of finite orbits local diffeomorphisms
were virtually unipotent,  i.e. 
the eigenvalues of  their linear parts were roots of unity.
The likely reason is revealed in Theorem \ref{thm_flow_fin_orb}: 
 time 1 maps that satisfy the finite orbits 
property have roots of unity eigenvalues.  
In this section we construct  a family of local diffeomorphisms that 
satisfy the finite orbits property but are non-virtually unipotent. 

\begin{defi}
	\emph{We say that $\lambda\in \C$ is a Cremer number
	if $\lambda$ is not a root of unity, but  
		\[ |\lambda|=1\ \mbox{and} \ \liminf_{m \to \infty} \sqrt[m]{|\lambda^{m}-1|} =0 . \]
	This equation is called  Cremer's condition.}
\end{defi}

Fix $n\geq 1$ and consider coordinates $x=(x_1,...,x_n)\in \C^n$ and $y\in \C.$
Given $j \in {\mathbb N}$ we denote by $[j]$ the unique natural number $j' \in \{1, \hdots, n\}$
such that $j-j'$ is a multiple of $n$. 
The proof of Theorem \ref{thm_irrat_bih_with_fin_orb} consists in building 
a convergent power series
	\begin{eqnarray}\label{a(x)}
	a(x_1,\hdots,x_n) =  \sum_{j=1}^{\infty}  \frac{(2 j x_{[j]})^{m_{j}}}{M_{j}^{m_{j}}}
	=   \sum_{j=1}^{\infty}   \left(\frac{2j}{M_j}\right)^{m_j}x_{[j]}^{m_j},
	\end{eqnarray}
	where  $(m_j)_{j \geq 1}$ is an increasing sequence of natural numbers and 
	$(M_j)_{j \geq 1}$ is a sequence of positive numbers that will be chosen to ensure that 
	\begin{eqnarray}\label{F(x,y)}
        F(x,y) = (\lambda x_1, \lambda x_2, \hdots, \lambda x_n, y + a(x_1, \hdots, x_n))
        \end{eqnarray}
       has finite orbits. 
 We need  auxiliary sequences  $(k_j)_{j \geq 1}$ and
$(r_j)_{j \geq 1}$ of natural numbers. 
They satisfy certain conditions that are provided by the following lemma.
\begin{lemma}
\label{lem:seq}
Let $\lambda$ be a Cremer number. There exist a sequence  $(M_j)_{j \geq 1}$ in ${\mathbb R}_{+}$ and sequences  $(m_j)_{j \geq 1}$, $(k_j)_{j \geq 1}$ and $(r_j)_{j \geq 1}$ in ${\mathbb N}$ such that, for any $j \in {\mathbb N}$, 
\begin{enumerate}
	\item[(C1)] $M_j := \frac{1}{\sqrt[{m_j}]{|\lambda^{m_j}-1|}}\ $ satisfies $M_j \geq 4j^{2}$; 
	\item[(C2)] $2^{m_j}  \geq 1 + j + \sum_{\ell=1}^{j-1} 2 (2\ell)^{m_\ell} j^{m_\ell}\  (2^{m_1}\geq 2)$;
	\item[(C3)] $\min (m_j, r_j) > \max (m_{j-1}, r_{j-1})$ if $j \geq 2$;
	\item[(C4)] $|\lambda^{k_j m_j} - 1| \geq 1$;
	\item[(C5)] $k_j \sum_{\ell = r_j}^{\infty} \frac{1}{2^{\ell}} < 1$. 
\end{enumerate} 

\end{lemma}	
\begin{proof}
Since $\lambda$ satisfies the Cremer condition, we can choose $m_1$ and $M_1$ such that the first three conditions hold, 
 where $(C3)$ is an empty condition. 
As $\lambda^{m_1}$ is not a root of unity, the sequence $(\lambda^{k m_1})_{k \geq 1}$ is dense on the unit circle; hence, there exists 
$k_1 \in {\mathbb N}$ such that the fourth condition holds for $j=1.$ Now we can define $r_1$ in such a way that the last condition holds for $j=1.$ Analogously, we can define $(M_2, m_2)$, $k_2$ and $r_2$ such that $(C1)-(C5)$ hold for $j=2.$ Indeed, we define the sequences $(M_j)_{j \geq 1}$,  $(m_j)_{j \geq 1}$,  $(k_j)_{j \geq 1}$ and $(r_j)_{j \geq 1}$ recursively for $j \in {\mathbb N}$.  
\end{proof}

\begin{rem}\label{rem_lambda_inv}
	\emph{ Notice that conditions $(C1)$, $(C2)$, $(C3)$, $(C4)$ and $(C5)$ still hold if we replace 
		$\lambda$ by  $\lambda^{-1},$ since $\overline{\lambda} = \lambda^{-1}$ implies 
		\[ |\lambda^{-n} -1| = |\overline{\lambda^{n}-1}| = | \lambda^{n} -1| \]
		for any $n \in {\mathbb Z}$.}
\end{rem}
\strut

\begin{lemma}\label{lema.a.e.inteira.}
       Consider the setting provided by Lemma \ref{lem:seq}. 
	Then $a(x)$ (cf. (\ref{a(x)})) is an entire function of ${\mathbb C}^{n}$. 
	Moreover, the map  $F(x,y)=(\lambda x, y+ a(x))$ 
	 is a holomorphic automorphism of $\C^{n+1}$ whose inverse is 
	\[  F^{-1}(x_1, \hdots, x_n,y)=(\lambda^{-1} x_1, \hdots, \lambda^{-1} x_n , 
	y - a(\lambda^{-1} x_1, \hdots, \lambda^{-1} x_n)) .   \]
 \end{lemma}

\begin{proof}
        Since $M_j \geq 4 j^{2}$ for any $j \geq 1$,  it follows that  
\[ \lim_{j \to \infty} \sqrt[m_j]{{\left( \frac{2j}{M_j} \right)}^{m_j}} =  
 \lim_{j \to \infty}  \frac{2j}{M_j} = 0 \]
       and hence $a(x_1, \hdots, x_n)$ is an entire function.
       We can verify directly that 
       $F^{-1}(x,y)= (\lambda^{-1} x, y - a(\lambda^{-1} x))$
	is the inverse of $F$.  
\end{proof}
\strut

Now note that if $A(x)= \sum a_{j_1 \hdots j_n} x_{1}^{j_1} \hdots x_n^{j_n}$ is a power series and 
$G(x,y)=(\lambda x_1, \hdots, \lambda x_n,  y+ A(x_1, \hdots, x_n)),$ then 
we can show that
\[G^k (x,y) = (\lambda^{k} x_1, \hdots, \lambda^{k} x_n,  y + (L_k A)(x_1, \hdots, x_n))\]
 for any $k \in {\mathbb N}$ by induction, 
where $L_k$ is the linear operator of the ring of convergent power series defined by 
\begin{eqnarray}\label{eq.prop.de.lk}
(L_k A) (x) &:=& A(x) + A (\lambda x) + \hdots + A(\lambda^{k-1} x) \\
&=& \sum_{j\in\N^n}  (1 + \lambda^{|j|} + \lambda^{2 |j|} + \hdots + \lambda^{(k-1) |j|}) a_{j_1 \hdots j_n} x_{1}^{j_1} \hdots x_n^{j_n}\nonumber  \\
&=& \sum_{j\in\N^n}  \frac{\lambda^{k |j|} -1}{\lambda^{|j|}-1} a_{j_1 \hdots j_n} x_{1}^{j_1} \hdots x_n^{j_n},\nonumber \end{eqnarray}
where $|j|=j_1 + j_2 + \hdots + j_n$.

\begin{lemma}\label{lema.cotas.para.a}
	\label{lem:orb_fin}
	Consider $j \geq 1$ and $x\in \C$ with 
	$\max_{1 \leq k \leq n} |x_k|   \leq j$ and $|x_{[j]}| \geq 1/j$. Then we have 
	\[j\leq|(L_{k_j} a)(x)|\leq 2(2j^2)^{m_j}+2^{m_j}-j.\] 
\end{lemma}
\begin{proof}
	First, let us study 
	 $\left|  L_{k_j} \frac{(2j x_{[j]})^{m_j}}{M_j^{m_j}} \right|.$  By (\ref{eq.prop.de.lk}) we have  
	 \[ \left|  L_{k_j} \frac{(2j x_{[j]})^{m_j}}{M_j^{m_j}} \right| =
	\left|  \frac{\lambda^{k_j m_j} -1}{\lambda^{m_j}-1}  \frac{(2j x_{[j]})^{m_j}}{M_j^{m_j}}   \right| =
	| (\lambda^{k_j m_j} -1) (2j x_{[j]})^{m_j}|,\]   
	where the second equality follows from the definition of $M_j.$ Thus, the choice of $k_j$ 
	allows us  to conclude that
	\[2^{m_j}\leq\left|  L_{k_j} \frac{(2j x_{[j]})^{m_j}}{M_j^{m_j}} \right|\leq 2(2j^2)^{m_j}\]
	since $\frac{1}{j} \leq |x_{[j]}| \leq j$.  Now, let us study  
	 \[ L_{k_j} \left( \sum_{l=1}^{j-1} \frac{(2l x_{[l]})^{m_l}}{M_{l}^{m_l}} \right). \] 
	 We obtain 
	\begin{eqnarray}
	\left| L_{k_j} \left( \sum_{l=1}^{j-1} \frac{(2l x_{[l]})^{m_l}}{M_{l}^{m_l}} \right)  \right|
	&=& \left|
	\sum_{l=1}^{j-1}  \frac{\lambda^{k_j m_l} -1}{\lambda^{m_l}-1} \frac{(2l x_{[l]})^{m_l}}{M_{l}^{m_l}} \right| \\
	&=&\left|
	\sum_{l=1}^{j-1}   (\lambda^{k_j m_l} -1) (2l)^{m_l} x_{[l]}^{m_l}   \right|\nonumber\\
	&\leq &\sum_{l=1}^{j-1} 2 (2l)^{m_l} j^{m_l} \nonumber\\
	&\leq&  2^{m_j} - (j+1)\nonumber 
	\end{eqnarray}  
 if $\max (|x_1|, \hdots, |x_n|) \leq j$, where the final 	inequality follows from condition (C2). 
	Finally, let us consider  
	\[ L_{k_j} \left( \sum_{l=j+1}^{\infty} \frac{(2l  x_{[l]})^{m_l}}{M_{l}^{m_l}} \right) . \]
	The condition (C1) implies $\frac{2l}{M_l}<\frac{1}{ 2j}$ for any $l >j.$ Therefore
	 $\max (|x_1|, \hdots, |x_n|) \leq j$  implies 
	\begin{eqnarray}
	\left|  L_{k_j} \left( \sum_{l=j+1}^{\infty} \left(\frac{2l}{M_l}\right)^{m_l}x_{[l]}^{m_l}\right) \right| 
	&\leq&\sum_{l=j+1}^{\infty} |1 + \lambda^{m_l} + \lambda^{2 m_l} + \hdots + \lambda^{(k_j-1) m_l}| \frac{1}{2^{m_l}} \\
	&\leq&   k_j  \sum_{l=j+1}^{\infty} \frac{1}{2^{m_l}}\nonumber\\
	&\leq&  k_j \sum_{l=r_j}^{\infty} \frac{1}{2^l} \nonumber \\
	&<& 1\nonumber
	\end{eqnarray}
	by conditions (C3) and (C5).
In particular,  by combining the previous estimates we get 
	\[j= 2^{m_j} - ( 2^{m_j} - (j+1) ) -1 <\left|(L_{k_j} a)(x)\right| <  2(2j^2)^{m_j}+2^{m_j}-j  \]  
	if $\max (|x_1|, \hdots, |x_n|) \leq j$ and $|x_{[j]}| \geq 1/j$. 
\end{proof}
\strut

The next lemma concludes the proof of Theorem \ref{thm_irrat_bih_with_fin_orb}.

\begin{lemma} \label{lem:pt2}
          Let $\lambda \in {\mathbb C}$ be a Cremer number and $n \geq 1$. Consider the function 
         $a(x)$ in (\ref{a(x)}) where $(M_j)_{j \geq 1}$ and $(m_j)_{j \geq 1}$ are provided by 
         Lemma \ref{lem:seq}.  Then the biholomorphism $F(x,y)= (\lambda x, y+a(x))$  
         has the finite orbits property in any set of the form $\C^{n} \times U,$ where $U$ 
         is a bounded open set in $\C$.
\end{lemma}
\begin{proof}
	Let $d$ be the diameter of $U.$  
	Fix $(x_{1,0}, \hdots, x_{n,0}, y_0) \in ({\mathbb C}^{n} \setminus \{0\}) \times U.$ Then there exists $j \in {\mathbb N}$ such that 
$\max (|x_{1,0}|, \hdots, |x_{n,0}|) \leq j$, $x_{[j],0} \geq 1/j$ and $j >d$.	
	 Lemma \ref{lema.cotas.para.a} implies that $F^{k_j} (x_0, y_0)$ does not belong to 
	 ${\mathbb C}^{n} \times U$. Notice that by Remark \ref{rem_lambda_inv} conditions $(C1)$, $(C2)$, $(C3)$, $(C4)$ and $(C5)$ still hold if we replace 
	 		$\lambda$ by  $\lambda^{-1}$.
Since  $F^{-1} (x,y) = (\lambda^{-1} x, y - a(\lambda^{-1} x))$, we have  
	\[ - a(\lambda^{-1} x) = \sum  (- \lambda^{-|j|}) a_{j_1 \hdots j_n} x_{1}^{j_1} \hdots x_n^{j_n}, \]
	that is, the monomials of $- a(\lambda^{-1} x)$ are obtained by multiplying those of $a(x)$ by complex numbers of modulus $1.$  In particular, the proof of Lemma \ref{lem:orb_fin} is still valid  for  $F^{-1}.$ One concludes that $F^{-k_j}(x_0, y_0) \not \in {\mathbb C}^{n} \times U.$ Therefore, $F$ has finite orbits in $({\mathbb C}^{n} \setminus \{0\}) \times U$. In the other hand, as $x=0$ is a fixed point curve of $F,$ it is clear that $F$ has the finite orbits property in ${\mathbb C}^{n} \times U.$  
	\end{proof}

\begin{theorem}\label{theorem.notable}\label{cor_notable}
Consider the hypotheses in Lemma \ref{lem:pt2}.
	Then the local diffeomorphism $F(x,y)=(\lambda x,y+a(x))$  is formally linearizable. 
	In particular, $F$ has a first integral of the form $y + b(x),$ where $b(x)$ is a divergent power series with $b(0)=0.$  Specifically,   $y-y_0 + b(x)=0$ 
defines a (divergent) formal invariant hypersurface through the point $(0,y_0)$
for every $y_0\in \C$.
\end{theorem}
\begin{proof}
	 The conjugacy equation 
	\[ (x, y + b(x)) \circ (\lambda x, y +a (x)) \circ (x, y-b(x)) = (\lambda x, y)  \]
	  is equivalent to  
	\[ b(x) - b(\lambda x)= a(x).\]
	 Setting  $b(x)=\sum_{j\in \N^n} b_jx^j\ \mbox{ and }\ a(x)=\sum_{j\in \N^n} a_jx^j,$ 
	 we see that $b(x) - b(\lambda x)= a(x)$ can be expressed as 
	 \[\sum_{j\in \N^n} (b_j -\lambda^{|j|} b_j-a_j)x^j=0.\]
	 Hence, the conjugacy equation has a solution 
	 \[b(x)=\sum_{j\in \N^n} \frac{a_j}{1-\lambda^{|j|}} x^{j}.\]
	Since $y$ is a first integral of $(x,y) \mapsto (\lambda x,y),$ the series $y + b(x)$ is a first integral of $F.$ Note that the series $b(x)$ is divergent, otherwise, $F$ and $(x,y) \mapsto (\lambda x,y)$ would be analytically conjugated, which is impossible, because $F$ has the finite orbits property whereas $(x,y) \mapsto (\lambda x,y)$  does not.
 \end{proof} 
\strut

We want to understand the non-virtually unipotent diffeomorphisms $F \in \diff{}{2}$ with finite orbits.  
 First, we focus on the arithmetic
properties of the non-root of unity eigenvalue. It is not casual that in our examples such 
an eigenvalue is well approximated by roots of unity. 

\begin{defi}
	\emph{A number $\lambda\in \C$ is  called a \textit{Bruno number} if there is a sequence $1<q_1<q_2<...$ of integers such that 
	\begin{equation}
	\label{eq.bruno.cond.}|\lambda|=1\ \mbox{ and }\sum_{k=1}^{\infty}\frac{1}{q_k}\log\frac{1}{\Omega_\lambda(q_{k+1})}<+\infty,
	\end{equation}
where  $\Omega_\lambda(m) =  \min_{2 \leq k \leq m} |\lambda^{k} - \lambda| $ for all $m\geq 2.$ Condition \ref{eq.bruno.cond.} is called \textit{Bruno Condition} and it is equivalent to the following (see \cite{bruno})
\[\sum_{k=1}^{\infty}\frac{1}{2^k}\log\frac{1}{\Omega_\lambda(2^{k+1})}<+\infty.\] 
Given $\lambda \in {\mathbb S}^1$ and $l \in {\mathbb N}$, we see that $\Omega_{\lambda^l}(m)\geq \Omega_{\lambda} ((m-1)l+1)$ for all $m\geq 2.$ Thus we can adjust (\ref{eq.bruno.cond.}) to conclude that if $\lambda$ is a Bruno number then so is $\lambda^{l}.$}
\end{defi}
\strut

\begin{pro}
\label{pro:nonbruno}
Consider  $F \in \diff{<\infty}{2}$ with  $D_0 F \not \in \diff{<\infty}{2}$.
Let $\lambda$ be the eigenvalue of $D_0 F$ that is not a root of unity. 
Then $\lambda$ is not a Bruno number. 
\end{pro}
\begin{proof}  
Suppose, aiming a contradiction, that one of the eigenvalues of $D_0 F$ is a Bruno number.
  By Theorem \ref{thm_at_least_one_root}, up to replace $F$ with a non-trivial
iterate $F^k$, we can suppose that $\mathrm{spec} (D_0 F) = \{1, \lambda \}$, where $\lambda$ is a Bruno number.

Up to a linear change of coordinates, we can suppose
$(D_0 F)(x,y)=(\lambda x, y)$. Note that the line $y=0$ is invariant by $D_0 F$.
Now, we apply a theorem of P\"{o}schel that relates
the invariant manifolds of $D_0 F$ and $F$ \cite{poschel1986}.
In our context, it determines a sufficient condition for the 
existence of a smooth analytic curve $\gamma$, invariant by $F$ , tangent to $y=0$ at $0$
 and such that 
$F_{|\gamma}$ is analytically conjugated to the rotation $x \mapsto \lambda x$. 
The P\"{o}schel  condition is 
\begin{equation}
\label{equ:poschel}
 \sum_{\nu \geq 0} 2^{-\nu} \log \omega^{-1} (2^{\nu +1}) < \infty 
\end{equation}
where we define 
\[ \omega (m) = \min_{2 \leq k \leq m} (|\lambda^{k} - \lambda|, |\lambda^{k}-1|).\]
Since $\omega(m) \geq \Omega_\lambda (m+1)$ for $m \geq 2$,  it follows that 
the property (\ref{equ:poschel}) is a consequence of the Bruno condition. 
Thus, the intended $\gamma$ exists and $F_{|\gamma}$ is an irrational rotation, 
contradicting the finite orbits property.   
\end{proof} 
\begin{cor}
\label{cor:pos}
Let $F \in \diff{}{2}$. Consider a formal invariant curve $\Gamma$ such that the multiplier of $F_{|\Gamma}$  
is a Bruno number. Then $F$ does not satisfy the finite orbits property.
\end{cor}	
\begin{proof}
Assume, aiming at a contradiction, that $F$ has finite orbits.  Let $\gamma (t)$ be a Puiseux parametrization
of $\Gamma$. Since $\Gamma$ is invariant, we have $F (\gamma (t)) =  (\gamma \circ h)(t)$ for some 
$h \in \fdiff{}{}$. We denote the multiplier of  $F_{|\Gamma}$ by $\mu$; it satisfies $\mu = h' (0)$.
We can suppose that $1 \in \mathrm{spec}(D_0 F)$ up to replace $F$ with some non-trivial iterate $F^{k}$
by Theorem \ref{thm_at_least_one_root}.
Note that the multiplier of $F_{|\Gamma}^{k}$ is equal to $\mu^{k}$. Since 
$\mu^{k}$ is a Bruno number, the hypothesis 
still holds for $F^{k}$ and $\Gamma$.
The tangent cone of $\Gamma$ is a subspace of eigenvectors of 
$D_0 F$, associated to an eigenvalue that we denote by $\lambda$. 
Moreover, we have $\mu^{m} = \lambda$, where $m$ is the multiplicity of $\Gamma$.
Since $\mu$ is a Bruno number, so is $\lambda$. Proposition  \ref{pro:nonbruno}  implies that $\lambda$ is not a Bruno
number, providing a contradiction.
\end{proof} 
\strut

Next, we see that the diffeomorphisms provided by Theorem \ref{thm_irrat_bih_with_fin_orb}
are archetypic examples of finite orbits diffeomorphisms 
$F \in \diff{}{2}$ such that $D_0 F$ has no finite orbits. 
Indeed, next result classifies the properties of such diffeomorphisms.  
\begin{pro}
\label{pro:pos}
Let $F \in \diff{<\infty}{2}$  with $\mathrm{spec}(D_0 F)=\{1, \lambda\}$
where $\lambda$ is not a root of unity.  Then $F$ satisfies the following properties:
\begin{itemize}
\item $\lambda$ is not a Bruno number;
\item $\mathrm{Fix} (F)$ is a smooth curve through the origin;
\item $F$ is formally conjugated to $(x,y) \mapsto (\lambda  x, y)$ by a formal diffeomorphism
that is transversally formal along $\mathrm{Fix}(F)$; 
\item there exists a divergent smooth invariant curve through any point $p \in \mathrm{Fix}(F)$.
\end{itemize}
\end{pro}   
\begin{proof} 
The eigenvalue $\lambda$ is a non-Bruno number by Proposition \ref{pro:nonbruno}.
Fix a sufficiently small domain of definition $B_{\rho} (0)$. 
Let $C$ be the connected component of the origin of  $\mathrm{Per}(F)$ 
(cf. equation (\ref{equ:perk})). It is complex analytic, has positive dimension 
and is contained in $\mathrm{Fix}(F^{m})$ for some $m \in {\mathbb N}$
by Theorem \ref{thm:structure}. The dimension of the germ of $C$ at 
the origin is less than $2$, since otherwise  
the germ of $F^{m}$ at $0$ is the identity map, 
contradicting $\lambda \not \in e^{2 \pi i {\mathbb Q}}$.
Therefore, the germ of $C$ at $0$ is an analytic curve $\gamma$.
Moreover, $\gamma$ is irreducible and smooth, since otherwise $F^{m}$ is tangent to the identity
by Lemma \ref{lema.da.unip.com.cur.inv}.
Since $F_{|\gamma}$ is a local biholomorphism in one variable with finite orbits, it has finite order. 
Therefore, its multiplier at $0$ is a root of unity and thus it is necessarily equal to $1$. 
Since the unique periodic tangent to the identity local diffeomorphism is the identity map, 
we deduce $\gamma \subset \mathrm{Fix}(F)$.  It is clear that the germ of $\mathrm{Fix}(F)$
at $0$ is contained in $C$ and hence the germs of $\mathrm{Fix}(F)$ and $\gamma$ at $0$
coincide.

Up to a change of coordinates in a neighborhood of the origin, we can assume 
$\mathrm{Fix}(F)= \{x=0\}$.
As a consequence $1 \in \mathrm{spec} (D_{(0,y)} F)$ for any $y$ in a neighborhood of $0$. 
We denote $\mathrm{spec}  (D_{(0,y)} F) = \{ 1, \lambda (y) \}$. 
The function $\lambda (y)$ is constant equal to $\lambda$ by the proof of Lemma 
\ref{lema.esp.const.}. We obtain that $F$ is of the form
\[ F(x,y) = \left( \lambda x + \sum_{j=2}^{\infty} a_{j}(y) x^{j}, 
y +  \sum_{j=1}^{\infty} b_{j}(y) x^{j} \right),  \]
where $a_{j+1}, b_{j}$ are defined in a common open neighborhood $U$ of $0$ in ${\mathbb C}$
for any $j \in {\mathbb N}$. 
We want to conjugate $F$ with $D_0 F$. In order to do it, 
we consider sequences  $(G_{2,j})_{j \geq 1}, (G_{1,j+1})_{j \geq 1}$ of diffeomorphisms  
of the form 
\[ G_{1,j+1}(x,y)= (x + c_{j+1}(y) x^{j+1}, y) \ \ \mathrm{and} \ \  G_{2,j}(x,y)= (x, y + d_{j} (y) x^{j}), \] 
where $d_{j}, c_{j+1} \in {\mathcal O}(U)$  for any $j \in {\mathbb N}$.
We define $F_{1,1} =F$, $F_{2,j} =  G_{2,j}^{-1} \circ F_{1,j} \circ G_{2,j}$ and
$F_{1,j+1} =  G_{1,j+1}^{-1} \circ F_{2,j} \circ G_{1,j+1}$ for $j \in {\mathbb N}$.
We want $F_{2,j}$ and $F_{1,j+1}$ to be of the form
\[ F_{2,j} (x,y)= (\lambda x + O(x^{j+1}), y + O(x^{j+1})), \ 
F_{1,j+1} (x,y) = (\lambda x  +  O(x^{j+2}), y + O(x^{j+1})) \]
for any $j \in {\mathbb N}$. Indeed, if $\alpha_{j+1} (y)$ is the coefficient of $x^{j+1}$ in 
$x \circ F_{2,j}$, it suffices to define 
$c_{j+1}(y)= \alpha_{j+1}(y)/(\lambda^{j+1} - \lambda)$ for $j \in {\mathbb N}$. 
Analogously, if $\beta_{j}(y)$ is the coefficient of $x^{j}$ in 
$y \circ F_{1,j}$, we have $d_{j}(y) = \beta_{j}(y)/(\lambda^{j}-1)$ for $j \in {\mathbb N}$.
The diffeomorphism
\[ H_{j} : = G_{2,1} \circ G_{1,2} \circ G_{2,2} \circ G_{1,3} \circ \hdots \circ G_{2,j-1} \circ G_{1,j} \]
conjugates $F$ with $F_{1,j}$ for $j \geq 2$.
By construction, it converges in the $(x)$-adic topology to some $H \in \fdiff{}{2}$, that 
is transversally formal along $x=0$ and satisfies $H^{-1} \circ F \circ H = D_0 F$.

Note that  $y \circ D_0 F \equiv y$ implies $(y \circ H^{-1}) \circ F \equiv y \circ H^{-1}$.
Since $y \circ H^{-1}$ is transversally formal along $x=0$, there exists a formal invariant curve 
$\gamma_{y}$ through $(0,y)$, that is transverse to $\mathrm{Fix}(F)$,  for any $y \in U$. 
We claim that $\gamma_{y}$ is divergent for any $y \in U$. 
Otherwise, there exists $y_0 \in U$ such that $\gamma_{y_0}$ is an analytic curve and 
since the multiplier of $F_{|\gamma_{y_0}}$ at $(0,y_0)$ is equal to $\lambda$, 
the diffeomorphism $F_{|\gamma_{y_0}}$ is non-periodic. This contradicts that the 
one dimensional diffeomorphism $F_{|\gamma_{y_0}}$ has finite orbits.
\end{proof}

\bibliographystyle{alpha}		
\bibliography{bibliografia}

\end{document}